\documentclass[11pt]{amsart}

\usepackage{amssymb,amsmath}
\usepackage[all]{xy}
\usepackage{url}
  \usepackage{color}
\usepackage{graphicx}
 
\usepackage[
            pdftex,
           colorlinks=true,
            urlcolor=blue,       % \href{...}{...} external (URL)
            filecolor=blue,     % \href{...} local file
            linkcolor=blue,       % \ref{...} and \pageref{...}
            citecolor=blue,         % \cite{...}
            pdftitle= {Title},
            pdfauthor={Author},
            pdfsubject={Subject},
            pdfkeywords={Key}
            pagebackref,
            pdfpagemode=None,
            bookmarksopen=true]
            {hyperref}
\usepackage{hyperref}
\usepackage{color}
\def\bfi{{\bf i}}
\def\Im{\mathop{\rm Im}\nolimits}

\def\gtilde{\tilde g}
\def\calTtilde{\tilde\calT}

\def\Gtilde{\tilde G}
\def\Htilde{\tilde H}
\def\htilde{\tilde h}
\def\Ftilde{\tilde F}
\def\Atilde{\tilde A}
\def\Dtilde{\tilde D}
\def\Etilde{\tilde E}
\def\gbar{\bar{g}}
\def\Gbar{\bar{G}}
\def\Abar{\bar{A}}

\def\sset{\subset}
\def\hur{{\mathcal{H}}}
\def\Z{\mathbb{Z}}
\def\Q{\mathbb{Q}}
\def\R{\mathbb{R}}
\def\F{\mathbb{F}}
\def\tensor{\otimes}
\def\semidirect{\rtimes}

\def\bfF{{\bf F}}
\let\iso\cong
\def\gend#1{\langle#1\rangle}

\def\orth{{\rm O}}

 \newtheorem{thm}{Theorem}[section]
 
 \newtheorem{lem}[thm]{Lemma}
 \newtheorem{corollary}[thm]{Corollary}

 \theoremstyle{definition}

 \newtheorem{remark}[thm]{Remark}

\numberwithin{equation}{section}

\newcommand{\bbA}{{\mathbb{A}}}
\newcommand{\bbC}{{\mathbb{C}}}

\newcommand{\bbH}{{\mathbb{H}}}

\newcommand{\bbR}{{\mathbb{R}}}
\newcommand{\bbP}{{\mathbb{P}}}

\newcommand{\bbQ}{{\mathbb{Q}}}

\newcommand{\bbZ}{{\mathbb{Z}}}

\newcommand{\GL}{\operatorname{GL}}

\newcommand{\SL}{\operatorname{SL}}

\newcommand{\PGL}{\operatorname{PGL}}

\newcommand{\Aut}{\operatorname{Aut}}

\newcommand{\Nef}{\operatorname{Nef}}

\newcommand{\Pic}{\operatorname{Pic}}

\newcommand{\diag}{\operatorname{diag}}

\newcommand{\cha}{\operatorname{char}}

\newcommand{\bsm}{\left(\begin{smallmatrix}}
\newcommand{\esm}{\end{smallmatrix}\right)}

\newcommand{\calC}{\mathcal{C}}
\newcommand{\calE}{\mathcal{E}}

\newcommand{\calO}{\mathcal{O}}

\newcommand{\calT}{\mathcal{T}}

\newcommand{\frakS}{\mathfrak{S}}
\newcommand{\frakA}{\mathfrak{A}}

\newcommand{\Num}{\operatorname{Num}}

\newcommand{\SO}{\operatorname{SO}}

\newcommand{\UC}{\operatorname{UC}}

\newcommand{\Bir}{\textup{Bir}}

\newcommand{\beq}{\begin{equation}}
\newcommand{\eeq}{\end{equation}}
\begin{document}

\title[The reflection group]{The  tetrahedron and  
automorphisms of Enriques and Coble surfaces of Hessian type}

\author{Daniel Allcock}
\email{allcock@math.utexas.edu}
\address{Department of Mathematics, University of Texas at Austin}
\urladdr{www.math.utexas.edu/$\sim$allcock}
\thanks{First author supported by Simons Foundation Collaboration Grant 429818.}

\author{Igor Dolgachev}
\email{idolga@umich.edu}
\address{Department of Mathematics, University of Michigan.}
\urladdr{http://www.math.lsa.umich.edu/$\sim$idolga}

\begin{abstract} 
  Consider a cubic surface satisfying the mild condition that it may be described in
  Sylvester's pentahedral form.  There is a well-known  Enriques or Coble surface $S$  
  with K3 cover  birationally isomorphic to the Hessian surface of this cubic surface.   
  We describe the nef cone and the $(-2)$-curves of~$S$.
  In the case of pentahedral parameters $(1,1,1,\discretionary{}{}{}1,\discretionary{}{}{}t\neq0)$ 
  we compute the
  automorphism group of~$S$. 
 For $t\neq1$ it is the semidirect product of
  the free product $(\Z/2)^{\ast 4}$ and the symmetric group $\frakS_4$. In the special case $t=\frac{1}{16}$ we study the action of
  $\Aut(S)$ on an invariant smooth rational curve $C$ on the Coble surface $S$.  We
  describe the action and its image, both geometrically and  arithmetically.
  In particular, we prove that $\Aut(S)\to\Aut(C)$ is injective in characteristic~$0$ and 
  we identify its image with the subgroup of $\PGL_2$
  coming from the 
  symmetries of a regular tetrahedron and the reflections across its facets.
\end{abstract}

\maketitle

\section{Introduction} 

%We work over a ground field $\Bbbk$ of characteristic${ }\neq2,3$.
Let $\UC(4)$ be the universal Coxeter group on 4 generators, i.e. a free product of four groups of order $2$. The permutation group $\frakS_4$ acts naturally on it.  Let $G = \UC(4)\rtimes \frakS_4$ be the semi-direct product. In this article we realize the group $G$ in several ways:

\begin{enumerate}
\item $G$ is the group of automorphisms of every Enriques and Coble surface 
  in a certain $1$-parameter family.
\item $G$ acts on an invariant smooth rational curve on a particular rational
  surface in this family,  faithfully when in characteristic~$0$.%$\cha\Bbbk=0$.
\item $G$ is a discrete group of motions of hyperbolic space $\bbH^9$.
\item $G$ is the (nondiscrete) group of isometries of $3$-dimensional Euclidean space
   generated by 
  the symmetries of a regular tetrahedron  and the reflections across its facets. 
\item $G$ is   the group of $\bbZ[\frac13]$-valued points of a 
  %simple 
  algebraic group scheme over $\bbZ$ coming from automorphisms of 
  Hamilton's 
  quaternion algebra.
\item $G$ is maximal among
  discrete subgroups with finite covolume in $\PGL_2({\bbQ}_3)$, where ${\bbQ}_3$ is
  the field of $3$-adic rationals. 
\end{enumerate}

The algebraic geometrical motivation for the study of this remarkable group is a question posed by Arthur Coble in the 1940s \cite{Coble}: given a group 
of birational automorphisms of an algebraic surface $S$ that leaves a rational curve $C$ on it invariant, what are the image and kernel of the restriction map to $\Bir(C)\cong \PGL_2$? Interest in this problem was recently resurrected in works  of Dinh and Oguiso \cite{Oguiso} and  Lesieutre \cite{Lesieutre}. In our situation, $S$ is the Coble surface obtained by blowing up all ten double points of a certain plane rational curve of degree 6 that admits $\frakS_4$ as its group of projective symmetries (see \cite[Sec.~5.4]{DolgachevSalem}). 
The curve $C\sset S$ is the proper inverse transform  of this sextic. 
Another description of $S$ begins with the Hessian surface $H$, of a
cubic surface having six Eckardt points and one ordinary node.  There is 
a birational involution of $\bbP^3$ that acts freely away from the node, and $S$ is the 
minimal resolution of the quotient of~$H$ by it. 
We will show that the group of automorphisms of $S$ is isomorphic to~$G$, via a homomorphism 
that identifies natural generators of $\Aut(S)$ with natural generators of $\UC(4)\semidirect\frakS_4$. This allows us to deduce that the restriction homomorphism $G\to\PGL_2(\Bbbk)$ is faithful when
working in characteristic~$0$.
The cubic surface that gives rise to this Coble surface $S$ is the $t=\frac1{16}$ 
member of the following one-parameter family of 
cubic surfaces
\begin{equation*}
y_0+\cdots+y_4=y_0^3+\cdots+y_3^3+ty_4^3=0
\end{equation*}
where $t\neq0$.
This family is also projectively isomorphic to the pencil of cubic surfaces with $\frakS_4$-symmetry (type V in Table 9.5.9 from \cite{CAG}\footnote{There is a misprint in the formula: the term $at_1t_2t_3$ must be added.}).
The minimal resolution $S_t$ of the quotient 
of the  Hessian surface  is a Coble surface if $t\in\{\frac{1}{4}, \frac{1}{16}\}$, or if
$t=1$ and $\cha\Bbbk=5$.
Otherwise it is an Enriques surface. 
%with automorphism group isomorphic to $\frakS_4$. 
We compute the group of automorphism of $S_t$ 
and obtain the amazing fact that it does not depend on the parameter $t\ne 1$ and is isomorphic 
to the group $G$. The exceptional case $t = 1$ corresponds to the Clebsch diagonal cubic surface.
In this case, both $S_t$ and the cubic surface have automorphism group $\frakS_5$,
and $S_t$ has type VI in Kondo's classification of complex Enriques surfaces with 
finite automorphism group \cite{Kondo}. 

Our strategy for working out $\Aut(S_t)$ is to use known automorphisms coming from the projections of the Hessian surface from its 10 nodes to build
a concrete model for the real nef cone $\Nef_\bbR(S_t)$. Then we use the shape of this
cone to show that these known automorphisms generate $\Aut(S_t)$.  
The nef cone is (the cone in $\bbR^{10}$ over) a polytope in hyperbolic $9$-space,
which usually has infinitely many facets.  (In the Coble case, $\Nef_\bbR(S_t)$ has dimension
  larger than~$10$.  But we identify a $10$-dimensional slice of it on which $\Aut(S_t)$
  acts faithfully.  By restricting attention to that slice, the Coble and
Enriques cases become uniform.) We are nevertheless able to give
a useful description of it in Theorem~\ref{Thmnefcone}, for surfaces arising from the Hessian surfaces of an arbitrary Sylvester non-degenerate cubic surface. It seems reasonable
to hope that this description will enable the computation
of $\Aut(S_t)$ in the same generality.  (For a  Sylvester nondegenerate nonsingular cubic surface with no Eckardt points this has been 
done by other methods by I. Shimada \cite{Shimada}.)

\medskip
\noindent
{\bf Acknowledgement.}
The authors thank Shigeru Mukai for stimulating discussions about the automorphism groups of Enriques surfaces arising from  Hessian surfaces during their stay at RIMS.

\section{Enriques and Coble surfaces of  Hessian type}
\label{SecBackground}

We work over an algebraically closed field $\Bbbk$ of characteristic $\ne 2,3$. Although some of the references in the paper refer to sources  which work over the field of complex numbers, all the proofs of what we need extend to our case. 
This section develops the situation of \cite{DolgachevSalem} in a manner which treats the Enriques and
Coble cases uniformly.  

Let $V_3 = \{F(x_0,x_1,x_2,x_3) = 0\}\subset \bbP^3$ be a cubic 
surface and $H\subset \bbP^3$ be its Hessian surface, given by the 
determinant of the matrix of second partial derivatives of $F$. 
We call $V_3$  \emph{Sylvester non-degenerate} if $F$ can be written 
as the sum of the cubes of five linear forms, any four of which are linearly independent.
A well-known theorem of Sylvester asserts that a 
general cubic surface is Sylvester non-degenerate and 
 that a Sylvester non-degenerate cubic surface uniquely determines the
linear forms, up to cube roots of unity and a common scaling factor (see \cite{CAG}, Theorem 9.4.1).
It is more convenient to work with certain multiples of the linear forms.
Namely, up to a common factor, 
they have unique nonzero multiples $L_0,\dots,L_4$ with the property
$\sum_{a=0}^4L_a=0$.  
We define $\lambda_0,\dots,\lambda_4$ by the relation $F=\sum_{a=0}^4\lambda_aL_a^3$.
The $\lambda_a$ are nonzero by Sylvester nondegeneracy.

Using the map $\bbP^3\to \bbP^4$ defined by
$$(x_0:x_1:x_2:x_3)\mapsto (L_0(x_0,x_1,x_2,x_3){:}\cdots{:} L_4(x_0,x_1,x_2,x_3))$$
we embed $V_3$ and $H$ into $\bbP^4$.  Their images are defined by
$$
\sum_{a=0}^4 y_a = \sum_{a=0}^4\lambda_ay_a^3 = 0
\qquad\hbox{and}\qquad
\sum_{a=0}^4 y_a = \sum_{a=0}^4\frac{1}{\lambda_a y_a} = 0
$$
respectively.
The last sum is shorthand for the quartic polynomial got by clearing denominators.
When speaking of $\bbP^3$ we will always mean the hyperplane $\sum_{a=0}^4y_a=0$.
The intersections of this hyperplane with the hyperplanes $y_a=0$ are called the
faces of the  \emph{Sylvester pentahedron}, and their pairwise and triple intersections
are called the edges and vertices of the pentahedron.
Each face contains four edges, which form the intersection of that face with the Hessian.
We write $L_{ab}$ for the edge $y_a=y_b=0$ and $P_{ab}$ for the vertex
defined by $y_c=0$ for all $c\neq a,b$.
$P_{ab}$ lies in $L_{cd}$ if and only if $\{a,b\}\cap\{c,d\}=\emptyset$.
The lines and vertices form  an abstract symmetric Desargues configuration $(10_3)$; see
Figure~\ref{sylvester}. 

\begin{figure}[ht]
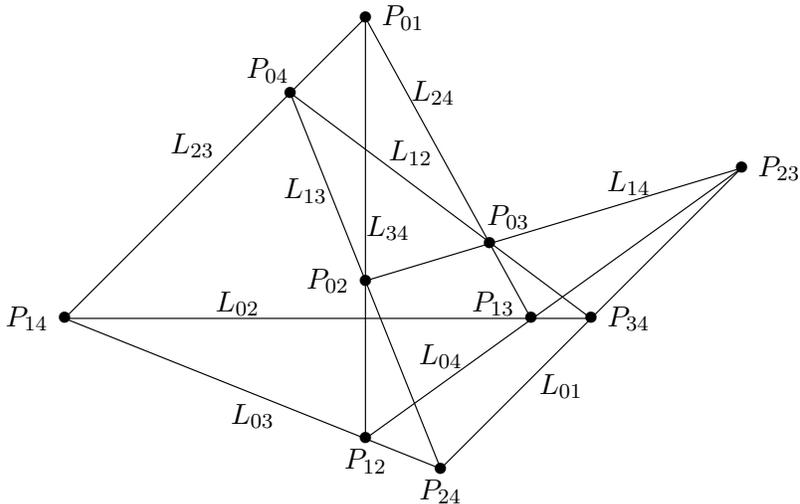

\xy (-20,-25)*{};
(0,0)*{};(40,40)*{}**\dir{-};(0,0)*{};(70,0)*{}**\dir{-};(0,0)*{};(50,-20)*{}**\dir{-};
(50,-20)*{};(90,20)*{}**\dir{-};(40,40)*{};(40,-16)*{}**\dir{-};(50,-20)*{};(30,30)*{}**\dir{-};
(40,40)*{};(62,0)*{}**\dir{-};(40,-16)*{};(90,20)*{}**\dir{-};(70,0)*{};(30,30)*{}**\dir{-};(90,20)*{};(40,5)*{}**\dir{-};
@={(0,0),(30,30),(40,40),(70,0),(90,20),(50,-20),(40,-16),(62,0),(40,5),(56.5,10)}@@{*{\bullet}};
(-5,0)*{P_{14}};(27,33)*{P_{04}};(45,40)*{P_{01}};(95,20)*{P_{23}};(75,0)*{P_{34}};(50,-23)*{P_{24}};
(40,-19)*{P_{12}};(35,5)*{P_{02}};(59,13.5)*{P_{03}};(57,2)*{P_{13}};
(17,23)*{L_{23}};(25,-13)*{L_{03}};(23,2)*{L_{02}};(66,-9)*{L_{01}};(43,12)*{L_{34}};(49,30)*{L_{24}};
(32,17)*{L_{13}};(75,18)*{L_{14}};(50,-5)*{L_{04}};(46,22)*{L_{12}};
\endxy
\caption{Sylvester pentahedron}\label{sylvester}
\end{figure}

One can check using partial derivatives that the vertices are ordinary nodes.
Every other point on an edge is smooth (since it is a smooth point of a planar section).
The Hessian may have additional singularities, but they are mild.  In light of
the following lemma, we call them the {\it new nodes}.  We will write $k$ for the number of
new nodes.

\begin{lem}
  \label{newsing}
  Let $V_3$ be a Sylvester non-degenerate cubic surface.
  The singular points of $V_3$ coincide with the singular points of
  $H$ away from the vertices of the pentahedron.  
  Each such point is an ordinary node of $V_3$ and an ordinary node of $H$.
\end{lem}

\begin{proof}
  First we find the singular points of $H$ away from the vertices of the pentahedron.
  Such a point $(y_0:\cdots:y_4)$ has all coordinates nonzero, so we may 
  take $\sum_{a=0}^4y_a=\sum_{a=0}^41/\lambda_ay_a=0$ as the defining equations.
  The method of Lagrange multipliers shows that the singular points of $H$ are
  given by the additional conditions
  $\lambda_ay_a^2=\lambda_by_b^2$ for all $a,b$.
  Regarding $(y_0,\dots,y_4)$ as a point of $\bbA^5$ lying over one of these singularities,
  the second derivative matrix  is a scalar multiple of the
  diagonal matrix $\diag[1/y_0,\dots,1/y_4]$.  
  We regard this as a bilinear form on the tangent space of $\bbA^5$ at this point.  To show that the singularity is an ordinary node, 
  it is enough to show that this form's restriction to the 
 coordinate-sum zero subspace has null space no larger than the radial direction in~$\bbA^5$.
 This is immediate: if $(c_0,\dots,c_4)$ lies in the null space, then orthogonality to the vectors like
 $(1,-1,0,0,0)$ forces $c_a/y_a=c_b/y_b$ for all $a,b$.

 The corresponding analysis for $V_3$ turns out to be exactly the same calculation. The only
 difference is  that
 $V_3$ has no singularities in the faces of the pentahedron.
 (In fact every singular point of any cubic surface is also a singular point of its Hessian.
  This follows easily from the alternate definition 
of $H$ as the discriminant surface of the web of polar quadrics of $V_3$; 
see \cite{CAG}, Proposition 1.1.17.) 
\end{proof}

Because the Hessian 
$H$ 
is a quartic surface whose singularities are ordinary nodes,
 its minimal resolution $X$  is a K3 surface. 
We write $E_{ab}$ for the 
exceptional curve over $P_{ab}$. Because $P_{ab}$ is an ordinary node, $E_{ab}$ is 
a $(-2)$-curve, meaning a smooth rational curve with self-intersection~$-2$.
We denote the proper transforms of the $L_{ab}$ by the same notation $L_{ab}$.
The ten curves $E_{ab}$ (resp.\ $L_{ab}$) are disjoint.
Also,
\begin{equation*}
  L_{ab}\cdot E_{cd} =
  \begin{cases}
    1&\hbox{if $\{a,b\}\cap \{c,d\} = \emptyset$,}\\
    0&\hbox{otherwise}
  \end{cases}
\end{equation*}

The birational involution of $\bbP^4$ defined by the formula
$$\sigma:(y_0:\cdots:y_4)\mapsto \Bigl(\frac{1}{\lambda_0y_0}:\cdots:\frac{1}{\lambda_4y_4}\Bigr)$$
restricts to a birational self-map of
$H$. 
This restriction 
is biregular on the complement of the faces 
of the pentahedron.  A calculation shows that the
fixed
points of $\sigma$ in this open set coincide with the new nodes.
We also write $\sigma$ for the corresponding self-map of $X$, which is biregular.
 (Every birational map from one K3 surface to another is biregular).
One can check that $\sigma$ swaps each $E_{ab}$ with the disjoint curve $L_{ab}$.
It follows that $\sigma$ acts freely on $X$
away from the exceptional divisors over the new nodes.  

The exceptional divisor over a new node is a $(-2)$-curve, got by blowing up once. 
Also, each new node
is an isolated fixed point under the action of $\sigma$ on $\bbP^4$, so $\sigma$ acts by negation
on the tangent space there.
It follows that  $\sigma$ acts trivially on the fibers over the new nodes.
Therefore the quotient surface $S=X/\langle\sigma\rangle$ is smooth. 
$S$ is the main object of interest in this paper.
We write $f:X\to S$ for the quotient map, 
$\calE$ for the sum of the exceptional
divisors over the new nodes, and 
$\calC$ for the branch divisor in $S$ (the image of $\calE$).
This branch divisor 
consists of  $k$ disjoint smooth rational curves with self-intersection~$-4$.
Also, 
 we write $U_{ab}$ for the common image  in $S$ of $E_{ab}$ and $L_{ab}$.
The $U_{ab}$ are $(-2)$-curves which intersect  according to the 
Petersen graph, whose  symmetry group is the symmetric group~$\frakS_5$. See Figure~\ref{petersen}.

\begin{figure}[ht]
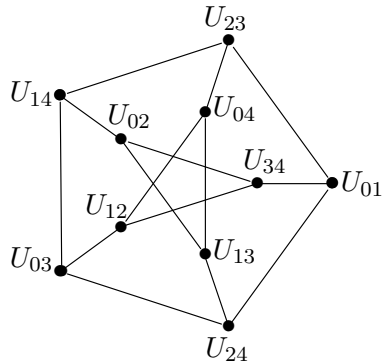
\label{peterseng}
\centering
\xy (-65,0)*{};(20,0)*{\bullet};(6.2,19)*{\bullet}**\dir{-};(-16.2, 11.7)*{\bullet}**\dir{-};(-16.2, 11.7)*{\bullet};(-16,-11.7)*{\bullet}**\dir{-};
(-16.2,-11.7)*{\bullet};(6.2,-19)*{\bullet}**\dir{-};(6.2,-19)*{};(20,0)*{\bullet}**\dir{-};
(10,0)*{\bullet};(-8.1,5.85)*{\bullet}**\dir{-};(-8.1,5.85)*{};(3.1, -9.5)*{\bullet}**\dir{-};(3.1, -9.5)*{};(3.1,9.5)*{\bullet}**\dir{-};
(3.1,9.5)*{};(-8.1,-5.85)*{\bullet}**\dir{-};
(-8.5,-5.85)*{};(10,0)*{}**\dir{-};   (10,0)*{};(20,0)*{}**\dir{-};  (3.1,9.5)*{};(6.2,19)*{}**\dir{-};
 (-16.2, 11.7)*{};(-8.1,5.85)*{}**\dir{-};(-16.2, -11.7)*{};(-8.2,-5.85)*{}**\dir{-};(6.2, -19)*{};(3.1,-9.85)*{}**\dir{-};
 (24,0)*{U_{01}};(5.8,22)*{U_{23}};(-20,12)*{U_{14}};(6,-22)*{U_{24}};(11,3)*{U_{34}};(7,10)*{U_{04}};(7,-9)*{U_{13}};(-7,9)*{U_{02}};
 (-10,-3)*{U_{12}};(-20,-10)*{U_{03}};
\endxy
\label{petersen}
\caption{The Petersen graph}\
\end{figure}

 We recall that a {\it Coble surface} means a smooth rational surface whose
anticanonical system is empty but whose bi-anticanonical system is not (see \cite{DZ}). 
They are important because they occur as degenerations of Enriques surfaces.

 \begin{lem}
   If there are no new nodes then $S$ is an Enriques surface.
   Otherwise, $S$ is a Coble surface, and 
   $\calC$ is the unique effective bi-anticanonical divisor on~$S$.
 \end{lem}

 \begin{proof}
   It is standard that the quotient of a K3 surface by a fixed-point-free involution is an Enriques
   surface.  So suppose $\sigma$ has fixed points.  
   The
   Hurwitz-type formula
   \begin{equation*}
     \textstyle
     0=2K_X=2(f^*(K_S)+\calE)=f^*(2K_S+\calC)
   \end{equation*}
   gives $\calC\sim-2K_S$.  So $\calC$ is a bi-anticanonical divisor.  There are no other
   effective bi-anticanonical divisors because
   the components of $\calC$ are 
   disjoint with
   negative self-intersection.  The only candidate for an effective 
   divisor in $|-K_S|$ is $\frac12\calC$,
   but this is not a divisor, so 
   $|-K_S|=\emptyset$.  It is well-known that the quotient of a K3 surface by an involution is either birationally an Enriques surface, or a rational surface, or a K3 surface. The latter case happens if and only if the involution has eight isolated fixed points. So, in our case $S$ must be a rational surface. We can also use the Castelnuovo rationality criterion: $H^1(S,\calO_S) = 0,   |2K_S| = \emptyset$ implies $S$ is rational. The first equality follows from $H^1(X,\calO_X) = 0$.
  \end{proof}

 One of the goals of this paper is to understand $\Aut(S)$.  We will use some elliptic fibrations
 $|G_{ab}|$
 of $S$ in order to show that $\Aut(S)$ acts faithfully on a certain lattice $\Lambda$ in $\Num_\Q(S)$.
 To construct these elliptic fibrations we begin with some elliptic fibrations of $X$.
 
Consider the pencil of planes in $\bbP^3$ that contain $L_{ab}$.  Each plane meets $H$ in $L_{ab}$
and a plane cubic curve.  The total transforms 
in $X$ of these residual cubics form a pencil  of elliptic curves, 
which we identify with the elliptic fibration $X\to \bbP^1$ it defines. 
We write 
$|\Gtilde_{ab}|$ for the corresponding linear system.  
By taking the plane to be the face $y_a=0$ of the Sylvester pentahedron, one
can express the linear system as
$$|\Gtilde_{ab}|=\bigl|L_{ac}+L_{ad}+L_{ae}+E_{bc}+E_{bd}+E_{be}\bigr|$$
where $\{c,d,e\}=\{0,\dots,4\}-\{a,b\}$.

One can check that if $z_1,z_2\in H$ have all coordinates nonzero, and lie on a plane containing
$L_{ab}$, then the same holds for $\sigma(z_1)$ and $\sigma(z_2)$.
It follows that our elliptic fibration of $X$ is  $\sigma$-invariant, so
it descends to an elliptic fibration of $S$.  
Furthermore, the action of $\sigma$ on the base $\bbP^1$ of $|\Gtilde_{ab}|$
is nontrivial, which implies that the corresponding linear system $|G_{ab}|$ on $S$ is given by
\begin{equation*}
  |G_{ab}|=\bigl|U_{ac}+U_{ad}+U_{ae}+U_{bc}+U_{bd}+U_{be}\bigr|.
\end{equation*}

Now we define the lattice $\Lambda$ on which $\Aut(S)$ will act faithfully. It was first introduced by S. Mukai in his unpublished work  on automorphisms of Coble surfaces, 
so we will refer to it as the \emph{Mukai lattice}. 
Let $C_1,\dots,C_k$ be the components of the bi-anticanonical divisor~$\calC$.  Let
$\Num(S)'$ 
be the lattice in $\Num_\Q(S)$ generated by $\Num(S)$ and $\frac12C_1,\dots,\frac12C_k$.
Let $\Lambda$ be the orthogonal complement of $C_1,\dots,C_k$ in $\Num(S)'$.

\begin{lem}
  \label{LemFaithful}
  $\Aut(S)$ acts faithfully on $\Lambda$, which is integral, even and 
  unimodular of signature $(1,9)$, and spanned 
  by the $U_{ab}$ and the 
\begin{equation}
  \label{Eqfab}
  \textstyle
f_{ab}=\frac12
\bigl(U_{ac}+U_{ad}+U_{ae}+U_{bc}+U_{bd}+U_{be}\bigr)
\end{equation}
\end{lem}

\begin{remark}
  \label{rkMukaiLattice}
Up to isometry there is a unique even unimodular lattice of signature $(1,9)$, often called $E_{10}$.
So $\Lambda\iso E_{10}$.  
This is no surprise because if there are no new nodes
then 
$S$ is an Enriques surface
and $\Lambda$ coincides with $\Num(S)$, which   
is well-known to be a copy of $E_{10}$.
Instead of $\Lambda$, it might seem simpler to consider the orthogonal complement of $C_1,\dots,C_k$ in
$\Num(S)$.  This turns out to be inconvenient because the isometry type of the
resulting lattice depends on the number of new nodes.  
(It is $E_{10}$ if there is one new node, and non-unimodular if
there is more than one.)
The Mukai lattice is the same in all cases.
\end{remark}

\begin{proof}
First we note that $\Aut(S)$ acts on $\Lambda$.  This is because
  $\Aut(S)$ preserves the unique member of $|-2K_S|$, namely $\calC$.  Therefore it permutes the 
  components of $\calC$.  So it preserves $\Num(S)'$ and their orthogonal complement 
  in $\Num(S)'$.  

  Next we prove $\dim\Lambda=10$.  We remarked above that the  $k=0$ case is 
  a  property of Enriques surfaces.  So suppose 
   $k>0$, in which case $S$ is rational.  Then
  $\dim(\Num(S))=10-K_S^2$.  (Blowing up a point
    increases both sides by~$1$, while
    blowing down a $(-1)$-curve decreases both sides by~$1$. So it is enough to check 
  equality  for $\bbP^2$.)  Since
  $K_S^2=\frac14\calC^2=-k$  we get $\dim(\Num(S))=10+k$.
  Since $\calC$ has $k$ components, we get 
  $\dim\Lambda=10$.  

  Now we show that $\Lambda$ is integral.  Any $x\in\Lambda$ may be expressed
  as $y+\sum_{i=1}^k r_iC_i/2$ where $y\in\Num(S)$ and $r_1,\dots,r_k\in\Z$.
  Rewriting this as $y=x-\sum r_i C_i/2$,
  expressing $x'\in\Lambda$ similarly, and using $\Lambda\perp C_i$ gives
  \begin{equation*}
    y\cdot y' = x\cdot x' + \sum r_ir_i'\frac{C_i\cdot C_i}{4}
    = x\cdot x'-\sum r_i r_i'
  \end{equation*}
  This proves $x\cdot x'\in\Z$.

  Next we show that $\Lambda$ contains the $f_{ab}$. 
  We already remarked that $\sigma$ acts nontrivially
  on the base $\bbP^1$ of each $|\Gtilde_{ab}|$.  Therefore it sends exactly two fibers of
  this elliptic fibration to themselves.  
  Choose one, call it $\Ftilde$, 
  and write $\Pi$ for the corresponding plane in $\bbP^3$.
  Now, $\Ftilde$ is the total transform of the cubic plane curve residual to
  $L_{ab}$ in $\Pi$, and we write it as $\Ftilde=A+E_1+\cdots+E_l$.  
  Here $E_1,\dots,E_l$ are the  $(-2)$-curves over the new nodes in $\Pi$ (if any), and $A$
  is the sum of the proper transform of the
  plane cubic curve and possibly the $(-2)$-curves over some vertices of the pentahedron.

  We will show that the corresponding fiber $F$ of $|G_{ab}|$ is something like 
  a double fiber.
  Write $B$ for the image of $A$ in $S$, and choose the labeling so that $C_1,\dots,C_l\subseteq S$
  are the images of $E_1,\dots,E_l$.    
  Because $\sigma$ acts freely on 
  $A$ away from $E_1,\dots,E_l$, which it fixes pointwise,
  we have $F=2B+C_1+\cdots+C_l$.
  That is, $f_{ab}=B+\frac12C_1+\cdots+\frac12C_l$, which lies in $\Num(S)'$. 
  %(This helps
  %motivate the definition of $\Num(S)'$.)
  As half the class of a fiber, $f_{ab}$ is orthogonal to every $C_1,\dots,C_l$, so $f_{ab}\in\Lambda$.
  (Remark: if $k=1$ then 
    $E_1$ lies in one of the $\sigma$-invariant fibers, so 
    we choose $\Ftilde$ to be the other one.  
    Then $f_{ab}=B$, so $f_{ab}$ lies in $\Num(S)$ not just $\Num(S)'$.
    This leads to Remark~\ref{rkMukaiLattice}'s isomorphism $\Lambda\iso E_{10}$ in the
  case of one new node.)

  The $U_{ab}$ lie in $\Lambda$ and 
  have inner product matrix of rank~$10$, so they span $\Lambda$ up to finite index.
  Therefore the signature of $\Lambda$ is the signature of the intersection pairing on the $U_{ab}$,
  which is $(1,9)$.  One can check that the lattice spanned by the $U_{ab}$ and $f_{ab}$
  is even unimodular.  By integrality, $\Lambda$ can be no larger than this unimodular lattice.
  So we have proven all our claims
  except for the faithfulness of the action.

  For this,
  suppose $g:S\to S$ is an isomorphism that 
  preserves each $U_{ab}$.  
  We know that $U_{01}$ meets each of $U_{23},U_{34},U_{24}$ in a single point.
  So each of these points is fixed.  Since these points are distinct and $U_{01}$ is isomorphic to 
  $\bbP^1$, we see that $g$
  fixes $U_{01}$ pointwise.  The same holds with other indices in place of $0,1$, so
  $g$ fixes every $U_{ab}$ pointwise.

  Next, $g$ preserves the elliptic pencil $|G_{01}|$.
  Each of the curves $U_{23},U_{24},U_{34}$ has 
  intersection number~$2$ with the class of the fiber, hence meets every fiber.
  So  $g$  preserves every fiber.  Furthermore, 
  the restriction of $g$ to a fiber preserves its intersection with
  $U_{23}\cup U_{24}\cup U_{34}$.  
  Regarding the generic point of $S$ is an elliptic curve over the function 
  field of $\bbP^1$, this shows that $g$ acts as the identity on a reduced $0$-dimensional subscheme
  of length~$6$.  It follows that $g$ is the identity.
\end{proof}

To each vertex $P_{ab}$ of the pentahedron
is associated a birational involution $\gtilde_{ab}$ of the Hessian surface.
Namely, projection away from $P_{ab}$ defines a dominant rational map $H\dasharrow\bbP^2$ of degree~$2$.
This realizes the function field of $H$ as a quadratic extension of that of $\bbP^2$, and
$\gtilde_{ab}$ is the nontrivial automorphism of this field extension.  
Again using the biregularity of birational maps between  K3 surfaces, we regard the
$\gtilde_{ab}$ as automorphisms of $X$.  

One can check that if $z_1,z_2\in H$ have all coordinates nonzero, and
lie on a line through $P_{ab}$, then the same holds for $\sigma(z_1)$ and $\sigma(z_2)$.
It follows that the $\gtilde_{ab}$
commute with $\sigma$ and therefore descend to automorphisms $g_{ab}$ of $S$.
The next step  is to examine how the $g_{ab}$ act on~$\Lambda$.  
The nature of $g_{ab}$ turns out to depend on whether the equality $\lambda_a=\lambda_b$
holds.  
(Remark: it is well-known that this holds if and only if $P_{ab}$ is an Eckardt point 
of the cubic surface $V_3$; see Example 9.1.25  of \cite{CAG}.)

To describe the $g_{ab}$ we must introduce some vectors $\alpha_{ab}\in\Lambda$ and
some isometries $t_{ab}$ of $\Lambda$.
We define 
\beq\label{Eqalphaab}
\alpha_{ab}=f_{ab}-U_{ab},
\eeq
where $f_{ab}$ is from \eqref{Eqfab}.
One can check that $\alpha_{ab}$ has self-intersection~$-2$.  
For later use we record
\begin{align}
  \label{EqUdotAlpha}
\alpha_{ab}\cdot U_{cd}={}&
  \begin{cases}
    2&\hbox{if $\{a,b\}=\{c,d\}$}\\
    0&\hbox{otherwise}
  \end{cases}
  \\
  \label{EqAlphaDotAlpha}
  \alpha_{ab}\cdot \alpha_{cd}={}&
  \begin{cases}
    1&\hbox{if $\{a,b\}\cap\{c,d\}$ is a singleton}\\
    0&\hbox{if $\{a,b\}\cap\{c,d\}$ is emtpy}
  \end{cases}
\end{align}
The group $\frakS_5$ of permutations of $0,\dots,4$ acts on $\Lambda$ by permuting the subscripts
of the $U_{ab}$ and $f_{ab}$.  We write 
$t_{ab}$ for the isometry of $\Lambda$ corresponding to  the transposition $(ab)$.

\begin{lem}[Corollary 4.3 of \cite{DolgachevSalem}]
  \label{L2.3}
  Suppose $\lambda_a\neq\lambda_b$.  Then $g_{ab}\in\Aut(S)$ 
  acts on $\Lambda\sset\Num(S)$ by the composition of
  $t_{ab}$ and the reflection in $\alpha_{ab}$.
\end{lem}

\begin{proof}
  We write $\Pi_{ab}\subseteq\bbP^3$ for the plane spanned by $P_{ab}$ and $L_{ab}$.
  We claim that the projection involution $\gtilde_{ab}$ sends $L_{ab}$ to neither itself nor $E_{ab}$.  
  In fact, 
  from $\lambda_a\neq\lambda_b$ it follows that $\Pi_{ab}\cap H$ consists of
   $L_{ab}$ and an
  irreducible plane cubic with a singularity at $P_{ab}$.  
  By definition, $\gtilde_{ab}$ exchanges  $L_{ab}$
  and this cubic (or rather its proper transform).  

  Next, looking at Figure \ref{sylvester} shows that 
  $\gtilde_{ab}$ permutes the lines $L_{cd}$ other than $L_{ab}$ via the action of the transposition 
  $(ab)\in \frakS_5$ on their subscripts.  This shows that $g_{ab}$ acts on the $U_{cd}$ other than $U_{ab}$
  as $t_{ab}$.  By considering their orthogonal complement, we see that $g_{ab}$ either fixes or negates
  $\alpha_{ab}$.  The first case would lead to $g_{ab}(U_{ab})=U_{ab}$, which contradicts the
  previous paragraph.  So
  $g_{ab}$ negates $\alpha_{ab}$.  We have described the action of $g_{ab}$ on a basis for
  $\Lambda\tensor\Q$, and one can now check agreement with the lemma.
\end{proof}

The reader may skip the rather technical
second paragraph of the next lemma. It is needed only for
proving Corollary~\ref{CorOrbitsOnCurves}, which itself  is not needed elsewhere in the paper.

\begin{lem}\label{lem:alphaab}
  Suppose $\lambda_a=\lambda_b$.  
  Then $\gtilde_{ab}$ 
  acts on $H$
  as the coordinate transposition $y_a\leftrightarrow y_b$,
  and $g_{ab}$ acts on $\Lambda\sset\Num(S)$ as $t_{ab}$.  
  Furthermore, $\alpha_{ab}\in\Lambda$ is represented by a rationally effective divisor.

  More specifically, if the plane $\Pi_{ab}\sset\bbP^3$ spanned by $L_{ab}$ and $P_{ab}$
  contains no new node, then $\alpha_{ab}$ is the class of a $(-2)$-curve.  Otherwise
  $\Pi_{ab}$ contains exactly two new nodes and
  $\alpha_{ab}=2[D]+\frac12[C_i]+\frac12[C_j]$, where $D$ is a $(-1)$-curve and $C_i,C_j$ 
  are the $(-4)$-curves corresponding to the two new nodes.
  %there exists an effective divisor on $S$ which maps to $\alpha_{ab}$ under
  %the orthogonal projection map $\Num(S)\to\Lambda$.
\end{lem}

\begin{proof}
  Suppose $x$ is a point of $H$, and let $x'$ be its image under the coordinate
  transposition $(ab)$.  Clearly it lies in $H$.  Also,  $x$, $x'$ and $P_{ab}$ are collinear.
  For $x$ generic, this shows that
  the definition of $\gtilde_{ab}$ is to swap $x$ with $x'$.  This proves our claims about $\gtilde_{ab}$
  and $g_{ab}$.

  %We write $\Pi_{ab}$ as in the previous proof.
  Under the hypothesis $\lambda_a=\lambda_b$,  we get $\Pi_{ab}\cap H=2L_{ab}+M+M'$
  where  $M,M'$ are lines through $P_{ab}$ (possibly coincident).  
  We keep the same notation $M,M'$ for their proper transforms in $X$.  
  %Working in  $H$, the lines
  %$M,M'$ do not pass
  %through any other vertices of the pentahedron.  
  %Working in $X$, it follows that
  %$M,M'$ are disjoint from all $E_{cd}$ except $E_{ab}$.
  The fiber $\Ftilde$ of
  $|\Gtilde_{ab}|$ corresponding to $\Pi_{ab}$ is
  \begin{equation*}
  \Ftilde=L_{ab}+M+M'+E_{ab}+E_1+\cdots+E_l
  \end{equation*}
  where $E_1,\dots,E_l$ are the exceptional divisors over the new nodes that lie in~$\Pi_{ab}$ (if any).
  Because $\sigma$ exchanges $L_{ab}$ with $E_{ab}$, it preserves this fiber.
  Arguing as in
  the proof of Lemma~\ref{LemFaithful} (with $A=L_{ab}+M+M'+E_{ab}$) shows that
  \begin{displaymath}
    \textstyle
    f_{ab}=B+\frac12C_1+\cdots+\frac12C_l
  \end{displaymath}
  where $B=U_{ab}+D$ with $D$ an effective divisor.
  This shows that $\alpha_{ab}=f_{ab}-U_{ab}$ is rationally effective.

  %Because $\sigma$ exchanges $L_{ab}$ with $E_{ab}$, it preserves this fiber, hence
  %preserves  $M+ M'$.  
  %Since $M+ M'$ is $\sigma$-invariant and misses all the $E_{cd}$ except $E_{ab}$, 
  %it therefore misses all the $L_{cd}$ except $L_{ab}$.  
  %Writing $C$ for the image of $M+M'$ in $S$,  it follows that $C$ is orthogonal to every
  %$U_{cd}$ except $U_{ab}$, with which it has intersection number~$2$.
  %That is, $C$ has
  %the same intersection pairings  as $\alpha_{ab}$ with all the $U_{cd}$.  This completes the proof.
  A more detailed analysis leads to two cases.  
  If $\Pi_{ab}$ contains no new nodes, then 
  $\sigma$ must act freely on $M\cup M'$.  This forces $M,M'$ to be distinct.
  So $\Ftilde$ is a cycle of four $(-2)$-curves.
  Taking the quotient by $\sigma$ shows that $f_{ab}$ is the sum of two $(-2)$-curves intersecting
  each other twice.  One is $U_{ab}$ and $D$ is the other.

  On the other hand, if $\Pi_{ab}$ contains a new node, then the planar section $M+M'+L_{ab}$
  of $H$ contains that node with multiplicity two.  So the node lies on both $M$ and $M'$.  Since
  $P_{ab}$ also lies on these lines, we get $M=M'$.  Since $\sigma$ has exactly two fixed points on $M$,
  $\Pi_{ab}$ contains exactly two new nodes.  So $\Ftilde=L_{ab}+2M+E_{ab}+E_i+E_j$ where $E_i,E_j$
  are the $(-2)$-curves in $X$ lying over these new nodes.  
  So
  $f_{ab}=U_{ab}+2D+\frac12 C_i+\frac12 C_j$,
  where $D$ is the image of $M$ in $S$.
  Since $D$ is clearly a smooth curve~$D$ that meets $U_{ab},C_i,C_j$ once each,
  the relation $f_{ab}^2=0$ forces $D^2=-1$.
\end{proof}

\section{The nef cone}
\label{SecNef}

We continue in the situation of section~\ref{SecBackground}.  In particular, $\lambda_0,\dots,\lambda_4\neq0$
are the parameters of the cubic surface in Sylvester pentahedral form,   and $S$ is the corresponding
Enriques or Coble surface.  We defined the Mukai lattice as a certain $10$-dimensional
lattice $\Lambda\sset\Num_\Q(S)$,  and showed that $\Aut(S)$ acts faithfully
on it.  
The main object in this section is the intersection of  the real nef cone $\Nef_\R(S)$
with $\Lambda_\R=\Lambda\tensor\R$.  In the next section we will use this
to compute
$\Aut(S)$ in the special case $(\lambda_0,\dots,\lambda_4)=(1,1,1,1,t)$.  But we impose no condition on the
$\lambda_a$ in this section.

The signature of $\Lambda$ is $(1,9)$, so the set of positive-norm lines in $\Lambda_\R$
forms a copy of hyperbolic $9$-space $H^9$.
In section~\ref{SecBackground} we 
defined twenty vectors $U_{ab},\alpha_{ab}\in\Lambda$.  
We define 
\begin{equation*}
  P=\bigl\{
  x\in\Lambda_\R
\bigm|
\hbox{$x\cdot U_{ab}\geq0$ and $x\cdot\alpha_{ab}\geq0$ for all $a,b$}
\bigr\}
\end{equation*}
We will call the $U_{ab}$ and $\alpha_{ab}$ the simple roots (of $P$).
Because the twenty classes $U_{ab}$ and $\alpha_{ab}$ have norm~$-2$, and 
their pairwise 
inner products are $0$, $1$ or~$2$ by \eqref{EqUdotAlpha}--\eqref{EqAlphaDotAlpha}, 
$P$ is a Coxeter polytope.
That is, it is a
fundamental domain for the Coxeter group generated
by the reflections across its facets.  

%\emph{Caution:} we will not be concerned with this Coxeter group but rather a 
%``twist'' of a subgroup of it.  

The Coxeter diagram $D$ of $P$ is rather complicated, but can be described as follows.
We follow the conventions that an ordinary edge means inner product~$1$
(indicating a dihedral angle $\pi/3$),
a double edge means inner product~$2$ (indicating parallelism at a point of $\partial H^9$), and the
absence of an edge means orthogonality.
The facets corresponding to the
$U_{ab}$ form a copy of the Petersen graph (Figure~\ref{petersen}).  The 
facets corresponding to the $\alpha_{ab}$
form a copy of the ``anti-Petersen'' graph, meaning that $\alpha_{ab}$ and $\alpha_{cd}$
are joined just if $U_{ab}$ and $U_{cd}$ are not.  Finally,
each $U_{ab}$ is joined to $\alpha_{ab}$ by a double edge.  
A drawing of this graph, due to Kondo \cite{Kondo}, appears in Figure~\ref{skondo6}.
\begin{figure}[ht]
\begin{center}
\includegraphics[scale=.4]{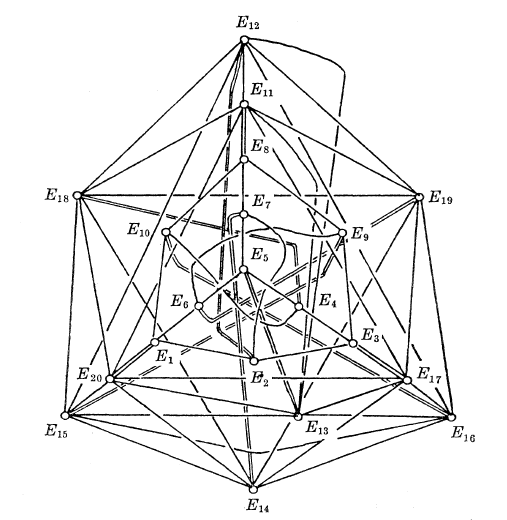}
\end{center}
\begin{center}
  \begin{tabular}{*{20}{c}}
$U_{01}$&$U_{02}$&$U_{03}$&$U_{04}$&$U_{12}$&$U_{13}$&$U_{14}$&$U_{23}$&$U_{24}$&$U_{34}$\\
$E_1$&$E_8$&$E_4$&$E_7$&$E_5$&$E_9$&$E_{10}$&$E_2$&$E_3$&$E_6$\\
\noalign{\medskip}
$\alpha_{01}$&$\alpha_{02}$&$\alpha_{03}$&$\alpha_{04}$&$\alpha_{12}$&$\alpha_{13}$&$\alpha_{14}$&$\alpha_{23}$&$\alpha_{24}$&$\alpha_{34}$\\
$E_{20}$&$E_{11}$&$E_{18}$&$E_{14}$&$E_{13}$&$E_{15}$&$E_{16}$&$E_{12}$&$E_{17}$&$E_{19}$\\
  \end{tabular}
\end{center}
\caption{Kondo's drawing of the Coxeter diagram of the polytope $P$, and how his labels correspond to ours.}
\label{skondo6}
\end{figure}
% Igor's original TeX for the correspondence
%$$(U_{01},U_{02},U_{03},U_{04},U_{12},U_{13},U_{14},U_{23},U_{24},U_{34})$$
%$$= (E_1,E_8,E_4,E_7,E_5,E_9,E_{10},E_2,E_3,E_6),$$
%$$(\alpha_{01},\alpha_{02},\alpha_{03},\alpha_{04},\alpha_{12},\alpha_{13},\alpha_{14},\alpha_{23},\alpha_{24},\alpha_{34})$$
%$$=(E_{20},E_{11},E_{18},E_{14},E_{13},E_{15},E_{16},E_{12},E_{17},E_{19}),$$

Using a criterion of Vinberg (Proposition~1 and Section~2.4 of \cite{Vinberg})
one can show that the  image of $P$
in real projective
space is a hyperbolic polytope with finite volume. 
We will pass freely between
$P$ and this hyperbolic polytope.
The center of $P$ is represented by $\Delta=\sum_{\{a,b\}}U_{ab}$,
which has inner product $1$ with each $U_{ab}$, inner product $2$ with each $\alpha_{ab}$,
and inner product~$10$ with itself.  

\begin{remark} 
  \label{RkPolytopeRemarks}
  $P$ is the nef cone of the Enriques surface $S$ of Type 
  VI in Kondo's classification of complex Enriques surfaces with finite automorphism group
  \cite{Kondo}.
  Its facets correspond to $(-2)$-curves on $S$, and Figure~\ref{skondo6} is the dual graph
  of these curves.  
  The finiteness of the hyperbolic volume of $P$ turns out to be
  essentially the same fact as the finiteness of $\Aut(S)$.
  This is because an Enriques surface has finite automorphism group if and only if it contains a 
  finite set of $(-2)$-curves which describe a finite-volume hyperbolic polytope.
  This particular surface
  $S$ has Hessian type, coming from the construction of Section~\ref{SecBackground}
  applied to the Clebsch diagonal surface.  That is, with $(\lambda_0,\ldots,\lambda_4) = (1,1,1,1,1)$.

  The polytope~$P$ also occurs in Shimada's paper \cite{Shimada}, where his
  Theorem~1.8 shows that $P$ is the union of 
  $2^{14}\cdot 3\cdot 5\cdot 7\cdot 17\cdot 31$ 
  fundamental chambers for the full reflection group $W_{237}$ of the lattice $\Lambda\iso E_{10}$.

For any Enriques surface $S$ of Sylvester-nondegenerate Hessian type, 
the class $\Delta\in\Pic(S)$ represents an ample \emph{Fano polarization} on $S$ that realizes the surface as a surface of degree $10$ in $\bbP^5$, the smallest possible projective embedding of an Enriques surface. Each such ample polarization taken with multiplicity $3$ is equal to the sum 
of ten isotropic nef divisors (in our case, the $f_{ab}$), all of whose pairwise 
intersection numbers are~$1$. Such polarization (maybe quasi-ample instead of ample) exists on any Enriques surface. In our case,
the associated projective
embedding sends each $U_{ab}$ to a line and any divisor representing $f_{\alpha}$ to 
a plane cubic curve.
\end{remark}
 
  \begin{thm}\label{Thmnefcone}  
  Let $G_0$ be the subgroup of
  $\Aut(\Lambda)$ generated by those $g_{cd}$ for which $\lambda_c\neq\lambda_d$.
  Then $\Nef_\R(S)\cap\Lambda_\R$ is the closure $Q$ of the union of the $G_0$-images of
  the polytope~$P$.  Furthemore, the facets of $Q$ correspond to the $G_0$-images of
  the $U_{ab}$ and of those $\alpha_{ab}$ for which $\lambda_a=\lambda_b$.
  
\end{thm}

\begin{proof}
  Write $Q$ for the closure of $\cup_{g\in G_0}\,g(P)$. (The only difference between the
    union and its closure is
  the limit set of $G_0$ in $\partial H^9$, which plays no role in our considerations.)
  The $U_{ab}$ are effective divisors, and the $\alpha_{ab}$ with $\lambda_a=\lambda_b$
  are rationally effective by Lemma~\ref{lem:alphaab}.  We call the corresponding facets of $P$ 
  the exterior facets.
  The remaining facets of $P$ correspond to the $\alpha_{cd}$ with $\lambda_c\neq\lambda_d$.
  Lemma~\ref{L2.3}
  shows that this facet is also a facet of $g_{cd}(P)$, which lies on the other side
  of the facet.
    It follows that the boundary of $Q$ in $H^9$ consists of 
    the $G_0$-images of the exterior facets of $P$.  This proves the last claim of the theorem,
    and shows that every facet of $Q$ is 
  orthogonal to a rationally effective divisor.  This implies $\Nef_\R(S)\cap\Lambda_\R\subseteq Q$.

  Suppose that the inclusion $\Nef_\R(S)\cap\Lambda_\R\subseteq Q$ were strict.
 Recall that 
 $\Nef_\bbR(S)$ is the intersection of the half-spaces $x\cdot B\geq0$ in $\Num_\bbR(S)$,
 where $B$ varies over
 the effective divisors with negative self-intersection.  
 
 So  $S$ has some such divisor $B$,
  whose orthogonal complement in $\Num_\R(S)$ meets the interior of $Q$.  
  Write $\beta$ for the projection of its class in $\Num(S)$ to $\Lambda$. 
  Without loss of generality we 
  replace $B$ by its  $G_0$-image having smallest possible $\beta\cdot\Delta$.
  This corresponds to  the hyperplane $\beta^\perp$ coming as close as possible 
  to $\Delta$ in hyperbolic space.  
  %By this choice, we know that no  $G_0$-image of $\Delta$
  %is closer to $\beta^\perp$ than $\Delta$ is.
  We claim that $\beta\cdot\alpha_{cd}\geq0$ when $\lambda_c\neq\lambda_d$.  
  Otherwise, the image $\beta+(\beta\cdot\alpha_{cd})\alpha_{cd}$ 
  of $\beta$ under the reflection in $\alpha_{cd}$ would have inner product 
  $\beta\cdot\Delta+2\beta\cdot\alpha_{ab} < \beta\cdot\Delta$ with $\Delta$.
  %
  %$(s_{\alpha_{cd}}(\beta),\Delta) = (\beta+(\beta,\alpha_{cd})\alpha_{cd},\Delta) = (\beta,\Delta)+2(\beta,\alpha_{ab}) < (\beta,\Delta)$.\marginpar{\ID{added explanation}}  Therefore, using Lemma \ref{L2.3}, we obtain that $g_{cd}\in G_0$ would also do this, which contradicts the choice of~$B$.

  Because $U_{ab}\in\Lambda$ we have $\beta\cdot U_{ab}=B\cdot U_{ab}\geq0$ for all $a,b$.  
When $\lambda_a=\lambda_b$, 
the same argument gives $\beta\cdot\alpha_{ab}=B\cdot\alpha_{ab}\geq0$
because $\alpha_{ab}$ is rationally effective.
  We have shown that $\beta$ has nonnegative inner products with the simple roots of $P$.
  That is, it lies in $P$.  Regarded as a hyperbolic polytope, $P$ has finite volume.
  Therefore, regarded as a cone in $\Lambda_\bbR$, it consists of norm${ }\geq0$ vectors.  This contradicts
  the hypothesis that $\beta^\perp$ meets the interior of $Q$ (or indeed any point of hyperbolic space).
\end{proof}

  \section{Automorphism groups}
\label{SecAuts}

  Our goal in this section is
to work out the automorphism groups of the Enriques and Coble surfaces $S$ constructed 
in Section~\ref{SecBackground}, for the parameters 
$(\lambda_0,\dots,\lambda_4)=(1,1,1,1,t)$, with $t\neq0$ as always.  
By Lemma~\ref{LemFaithful}, $\Aut(S)$ acts faithfully on the lattice $\Lambda\sset\Num_\Q(S)$.
And Theorem~\ref{Thmnefcone} describes the invariant cone $Q=\Nef_\R(S)\cap\Lambda_\R$
fairly explicitly.  The idea in this section is to use the shape of this cone 
to show that $\Aut(S)$ is generated by the known automorphisms~$g_{ab}$.

In the particular case $t=1$, the cubic surface is the Clebsch diagonal cubic surface.
It is smooth except in characteristic~$5$, in which case it has an ordinary node
at $(1:1:1:1:1)$.  Regardless of whether there is a node,
Theorem~\ref{Thmnefcone} shows that $Q=P$.  The isometry group of $P$ is the automorphism group of
its Coxeter diagram, which is just the obvious group $\frakS_5$
arising from permuting the pentahedral coordinates $y_0,\dots,y_4$.
So $\Aut(S)=\frakS_5$.  
Over $\bbC$, this Enriques surface has type VI in Kondo's classification of complex Enriques
  surfaces with finite automorphism group;  see  Remark~\ref{RkPolytopeRemarks}.

%Although it does not affect the analysis, 
We remark on the other special surfaces
in the family.
If  
$t\notin\{\frac1{16},\frac14\}$ then the Hessian surface $H$ is smooth away from the
vertices of the Sylvester pentahedron.  In this case the Cremona involution
$\sigma$ acts freely on $X$, so $S$ is
an Enriques surface.  
When $t=\frac1{16}$ there is one new node,
at $(1,1,1,1,-4)$.  When $t=\frac14$ there are four new nodes, at the images of
$(1,1,1,-1,-2)$ under the permutations of the first four coordinates. 

%We write $Q$ for $\Nef(S)_\R)\cap \Lambda_\bbR$.\marginpar{changed notation}  We refer to \cite{DolgKondo}, Chapter 10 for the following. \marginpar{I can supply the proof but it is rather long}

%\begin{lem}\label{dolgkondo} A divisor class $x\in \Nef(S)_\R)\cap \Lambda_\bbR$  if and only $x\cdot r \ge 0$ for any irreducible effective root. 
%\end{lem}

% As follows from Lemma \ref{dolgkondo}, 
%$Q$ is completely described by 
%the inequalities
%$x\cdot r\geq0$ for all irreducible effective roots. We know from Remark \ref{rmk} that not all of these roots are defined by $(-2)$-curves if one of the planes contains a new singular point.
%Thus $Q$ is also a Coxeter polytope (but may have infinitely many facets) that contains $P$.
We will need to
understand the cusps of the finite-volume hyperbolic
polytope $P$ introduced in Section~\ref{SecNef}. 
They correspond to parabolic subdiagrams of rank $8$ of the Coxeter diagram $D$ of $P$, and hence 
to elliptic pencils on Kondo's surface mentioned above. 
By inspection of the graph (or referring to Table 2 on page 274 in \cite{Kondo}), we obtain the following.   

\begin{lem}
  \label{LemCusps}
  There are four $\frakS_5$-orbits of cusps of $P$, corresponding to
subdiagrams $\Atilde_5\Atilde_1\Atilde_2$, $\Etilde_6\Atilde_2$, $\Dtilde_5\Atilde_3$ 
and $\Atilde_4\Atilde_4$ of $D$. 
In each case the  first  listed component
lies in the Petersen graph rather than the anti-Petersen graph, hence corresponds to
some set of divisors $U_{ab}$.
\qed
\end{lem}

\begin{thm} 
  \label{ThmAutGenerators}
  Let $S$ be the Enriques or Coble surface arising 
  from the Hessian of a Sylvester non-degenerate cubic surface
  with parameters 
  \begin{equation*}
    (\lambda_0,\dots,\lambda_4)=(1,1,1,1,t)
  \end{equation*}
  Then $\Aut(S)$ is
  generated by the ten involutions $g_{ab}$.
\end{thm} 

\begin{proof}
  We treated the $t=1$ case above, so we suppose $t\neq1$.
Theorem~\ref{Thmnefcone} shows that $Q$ is (the closure of) the union of the $G_0$-translates
of  $P$.  
Here $G_0$ is the subgroup of
$\Aut(S)$ generated by those $g_{cd}$ with $\lambda_c\neq\lambda_d$.
We will also write $G$ for the group generated by all ten $g_{ab}$.
  Suppose $g\in\Aut(S)$. Our strategy is to replace $g$ by
  its compositions with elements of $G$, ultimately leading to the conclusion
  that $g$ is the identity.
  The main idea of the proof is to show that the tessellation of $Q$ by copies of $P$ is
  intrinsic, so that $g$ preserves it.  
  We will do this by comparing the cusps of $Q$ to the cusps of $P$.
  As in the proof of theorem~\ref{Thmnefcone} we use the term ``exterior facets'' 
  for the facets of $P$ corresponding to the
  $U_{ab}$ and to those $\alpha_{ab}$ with $\lambda_a=\lambda_b$.
  We call the other  facets ``interior''.

  The fact that $\lambda_4$ is different from $\lambda_0,\dots,\lambda_3$ breaks
  the $\frakS_5$ 
  symmetry.  Each $\frakS_5$-orbit of cusps of $P$ breaks up into several orbits under the subgroup
  $\frakS_4$
  acting on the indices $\{0,1,2,3\}$.
  We will focus on the cusps of type $\Etilde_6\Atilde_2$.  There is one such cusp for 
  every ordered pair $a,b$ of distinct elements of $\{0,\dots,4\}$.  Namely,
  $\Etilde_6(a,b)$ has $U_{ab}$ as the branch  node, $U_{cd},U_{de},U_{ec}$ as its neighbors,
  and $U_{ae},U_{ac},U_{ad}$ as the end nodes.  
  Here $\{c,d,e\}=\{0,\dots,4\}-\{a,b\}$. The nodes of $D$  not joined to $\Etilde_6(a,b)$
  are $\alpha_{bc},\alpha_{bd},\alpha_{be}$, which form an $\Atilde_2$ diagram.  The corresponding
  facets of $P$ can be either exterior or interior.  Under
  $\frakS_4$ there are three orbits of  $\Etilde_6\Atilde_2$ diagrams, namely
  \begin{center}
    \begin{tabular}{ccc}
      &exterior $\Atilde_2$ facets&interior $\Atilde_2$ facets
      \\
      $\Etilde_6(4,0)$&$\alpha_{01},\alpha_{02},\alpha_{03}$&none
      \\
      $\Etilde_6(0,1)$&$\alpha_{12},\alpha_{13}$&$\alpha_{14}$
      \\
      $\Etilde_6(0,4)$&none&$\alpha_{14},\alpha_{24},\alpha_{34}$
    \end{tabular}
  \end{center}
 
  It is easy to write down a vector in $\Lambda$ representing the cusp
  corresponding to a given $\Etilde_6(a,b)$,  namely the null vector
  \begin{equation}
    \label{EqCusp}
    \nu_{a,b}=
    3 U_{ab} + 2(U_{cd}+U_{de}+U_{ec})+(U_{ae}+U_{ac}+U_{ad})
  \end{equation}

  Set $\nu=\nu_{4,0}$.  
  Since all the facets of $P$
  incident to $\nu$ are exterior facets, $\nu$ is also a cusp of $Q$.  
  Since $g(\nu)$ is a cusp of $Q$, it is a cusp
  of some $G_0$-translate of $P$.  By replacing $g$ by its composition with a suitable element
  of $G_0$, we may therefore suppose without loss that $g(\nu)$ is a cusp of $P$.  Some of the
  facets of $g(\nu)$ as a cusp of $Q$ might not be facets of $P$, because $Q$ might
  contain several translates of $P$ that are incident to $g(\nu)$.
  Nevertheless,
  every exterior facet of $P$ that contains $g(\nu)$ will also be a facet of $Q$.  Lemma~\ref{LemCusps}
  shows that  these facets of $P$ account for
  a subdiagram $\Atilde_5$, $\Etilde_6$, $\Dtilde_5$ or $\Atilde_4$ of $g(\nu)$'s diagram in $Q$.
  The cases $\Atilde_5$, $\Dtilde_5$ and $\Atilde_4$ are incompatible with the fact that
  the diagram of $\nu$ in $Q$ is $\Etilde_6\Atilde_2$.
  It follows that the $\Etilde_6$ component, of the diagram of $\nu$ 
  as a cusp of $Q$, consists of exterior walls
  of $P$.  By replacing $g$ by its composition with some element of $\frakS_4\sset G$,
  we may suppose without loss that $g(\nu)$ is one of the cusps
  $\nu_{4,0}$, $\nu_{0,1}$ or $\nu_{0,4}$ of~$P$.

  We claim that the first of these three cases holds, which is to say that $g$ fixes the
  cusp $\nu_{4,0}$.  
  One can check that $g_{41}$ exchanges the other two cusps, by using the explicit
  formula \eqref{EqCusp} for vectors representing them.  
  So in the case $g(\nu)=\nu_{0,1}$ we may replace $g$ by its composition with $g_{41}$,
  reducing to the case that $g(\nu)=\nu_{0,4}$. In this case we will derive a contradiction.
  The facets $\alpha_{41},\alpha_{42},\alpha_{43}$ of $P$ are incident to $\nu_{0,4}$, but they
  are interior facets.  Therefore $Q$ also contains the images of $P$ under $g_{41},g_{42},g_{43}$.
  We will focus on $g_{41}(P)$.  
  We already noted that $\alpha_{12}$ is an exterior facet of $P$ incident to $\nu_{0,1}$,
  and that  $g_{41}$ sends 
  $\nu_{0,1}$ to $\nu_{4,0}$.  Therefore $g_{41}(\alpha_{12})=\alpha_{41}+\alpha_{12}$
  is a facet of $Q$ incident to $\nu_{0,4}$.  
  Since $\nu_{0,4}$ is
  invariant under permutations of the indices $1,2,3$, the remaining two facets of $Q$
  at $\nu_{0,4}$ are $\alpha_{42}+\alpha_{23}$ and $\alpha_{43}+\alpha_{31}$.
  
  In summary, the facets of $Q$ at $\nu_{4,0}$ are $\Etilde_6(4,0)$
  and $\alpha_{01},\alpha_{02},\alpha_{03}$, while the facets of $Q$ at $\nu_{0,4}$
  are $\Etilde_6(0,4)$ and $\alpha_{41}+\alpha_{12},
  \alpha_{42}+\alpha_{23}, \alpha_{43}+\alpha_{31}$.
  Under our assumption
  $g(\nu_{4,0})=\nu_{0,4}$, we see that $g$ sends $\Etilde_6(0,4)$ to $\Etilde_6(4,0)$ 
  and $\{\alpha_{01},\alpha_{02},\alpha_{03}\}$ to
  $\{\alpha_{41}+\alpha_{12},
  \alpha_{42}+\alpha_{23}, \alpha_{43}+\alpha_{31}\}$.
  This contradicts
  \begin{align*}
    \alpha_{01}+\alpha_{02}+\alpha_{03}&
    \textstyle
    { }=\frac12\nu_{4,0}
    \\
    %\noalign{is incompatible with}
  (\alpha_{41}+\alpha_{12})+
  (\alpha_{42}+\alpha_{23})+ (\alpha_{43}+\alpha_{31})&{ }=
  \textstyle\phantom{\frac12}
  \nu_{0,4}
  \end{align*}
  (The fact that the right side is twice as large in the second line
     is the numerical manifestation of the idea
    that the cusp
  of $Q$ at $\nu_{0,4}$ is ``twice as big'' as the cusp of $Q$ at $\nu_{4,0}$.  One should
visualize an equilateral triangle of edge length $2$, divided into four equilateral 
triangles of edge length~$1$.)
  
  We have reduced to the case that $g$ fixes $\nu_{4,0}$.
  The center of the polytope $P$ is represented by $\Delta= \sum_{\{a,b\}}U_{ab}$.
  It has inner product $1$ with each $U_{ab}$, inner product $2$ with each $\alpha_{ab}$,
  and inner product~$10$ with itself. We may characterize it in terms of the cusp $\nu_{4,0}$
  as follows.  It is the unique norm~$10$ element of $\Lambda$ that is effective,
  has inner product~$1$ with each root
  of the $\Etilde_6$ diagram at $\nu_{4,0}$, and has inner product~$2$ with each root of the $\Atilde_2$
  diagram there.   Therefore $g$ preserves $\Delta$.  
  
  If $r\in\Lambda$ is the class of a $(-2)$ curve, then $r\cdot\Delta\geq1$, with equality if and only if
  $r$ equals some $U_{ab}$.  
  So $g$ permutes the $U_{ab}$, hence the $\alpha_{ab}$, hence preserves~$P$.
  Since it also preserves $Q$, $g$
  permutes the interior facets $\alpha_{a4}$ of $P$ amongst themselves.  By replacing $g$ by
  its composition with an element of $\frakS_4\sset G$, we may suppose that it preserves each
  of them.  The only automorphism of the Coxeter diagram $D$ with this property is
  the identity.  So $g$ acts by the identity on $\Lambda$.  Then Lemma~\ref{LemFaithful} shows
  that $g$ is the identity, completing the proof.
\end{proof}

\begin{corollary}
  \label{CorOrbitsOnCurves}
  For $t\neq\frac14$, there are two orbits of $\Aut(S)$ on the set of $(-2)$ curves, 
  with orbit representatives $U_{01},\alpha_{01}$.  For 
  $t=\frac14$, no $\alpha_{ab}$ is represented by a $(-2)$-curve, and every $(-2)$-curve
  is $\Aut(S)$-equivalent to $U_{01}$. 
\end{corollary}

\begin{proof}
  Consider the simple roots of $Q$.  By Theorem~\ref{Thmnefcone} these are
  the $\Aut(S)$-images of the $U_{ab}$ and of those $\alpha_{cd}$ for which $\lambda_c=\lambda_d$.
  Each $U_{ab}$ is a $(-2)$-curve. And by Lemma~\ref{lem:alphaab},
  each of these $\alpha_{cd}$ is either represented by a $(-2)$-curve, or else has the
  form $\alpha_{cd}=2[D]+\frac12[C_i]+\frac12[C_j]$ where $D$ is a $(-1)$-curve and $C_i,C_j$
  are two components of the bi-anticanonical divisor (and in particular are $(-4$)-curves).  
  The latter case occurs if and only
  if 
  the plane $\Pi_{cd}\sset\bbP^3$ containing $L_{cd}$ and $P_{cd}$ contains two new nodes.
  This happens only when $t=\frac14$, and then it happens for all $c,d\neq4$.
  
  Now suppose $R$ is a $(-2)$-curve in $S$.  The adjunction formula forces $R\cdot K_S=0$, so
  $R$ misses the components $C_1,\dots,C_k$ of the bi-anticanonical divisor.  It follows that
  the class of $R$ in $\Num(S)$ lies in $\Lambda$.  We claim that $R$ equals one of the
  $(-2)$-curves from the previous paragraph.  Otherwise it would have inner product${ }\geq0$ with
  all the simple roots of~$Q$.  But these define (the intersection with $\Lambda_\bbR$ of)
  the nef cone, so $R$ would be nef, contrary to $R^2<0$.
  Therefore either
  $R$ is an $\Aut(S)$-image of some $U_{ab}$, or else $t\neq\frac14$ and $R$ is an
  $\Aut(S)$-image of some $\alpha_{cd}$ with $\lambda_c=\lambda_d$.
  
  When $t=1$ we have seen that $\Aut(S)=\frakS_5$, whose orbits on the simple roots of $Q=P$
  have representatives
  $U_{01}$, $\alpha_{01}$.    So suppose $t\neq1$.
  Using $\frakS_4\sset\Aut(S)$ shows that
  $R$ is $\Aut(S)$-equivalent to $U_{01}$, $U_{04}$ or $\alpha_{01}$, with the last case only
  possible when $t\neq\frac14$. 
  We claim $g_{14}(U_{04})=U_{01}$.  
  To see this, recall that $g_{14}$ acts by the transposition $(14)$ on subscripts, followed by
  reflection in $\alpha_{14}$.  The transposition sends $U_{04}$ to $U_{01}$, which is orthogonal to
  $\alpha_{14}$.  This proves our claim.

  All that remains to prove is that $U_{01}$ and $\alpha_{01}$ are not $\Aut(S)$-equivalent.
  It suffices to show that 
  they are not equivalent under the much larger group
  $W\semidirect\Aut(P)$, where $W$ is the Coxeter group of the polytope $P$.
  This is equivalent to the non-conjugacy in $W\semidirect\Aut(P)$
  of the reflections of $\Lambda$ corresponding to $U_{01}$ and
  $\alpha_{01}$.
  To prove this non-conjugacy, map $W\semidirect\Aut(P)$ to $(\Z/2)^2$ by sending the reflections 
  corresponding to the $U_{ab},\alpha_{ab}$ to
  $(1,0),(0,1)\in(\Z/2)^2$ respectively, and sending $\Aut(P)\iso\frakS_5$ to the identity.
  That this defines a homomorphism can be checked by using the standard presentation of $W$
  in terms of the Coxeter diagram of~$P$.
  Since $(1,0)$ and $(0,1)$ are not conjugate in $(\Z/2)^2$, the reflections in 
  $U_{01}$ and
  $\alpha_{01}$ cannot be conjugate in $W\semidirect\Aut(P)$.
\end{proof}

\begin{thm} 
  \label{ThmAutGp}
  Suppose $S$ is as in theorem~\ref{ThmAutGenerators}, with $t\neq1$.  Then
  \begin{equation}
    \label{eqGroupEquality}
    \Aut(S)=\bigl((\Z/2)*(\Z/2)*(\Z/2)*(\Z/2)\bigr)\semidirect\frakS_4
  \end{equation}
  where the factors of the free product are generated by 
  $g_{04},g_{14},g_{24},g_{34}$.  
  Also, the group $\frakS_4$ is generated by the $g_{ab}$ with $a,b\neq4$,
  and permutes the factors of the free product in the obvious way.
\end{thm}

\begin{proof}
  Because the action of $\Aut(S)$ on $\Lambda$ is independent of the parameter $t$ and the characteristic
  of $\Bbbk$,
  it suffices to prove this  for $t=\frac1{16}$ in characteristic~$0$.
  In this case $k=1$, so the unique bi-anticanonical divisor of $S$ is a 
  smooth rational curve.  This yields a homomorphism $\Aut(S)\to\PGL_2(\Bbbk)$.  We show in 
  Theorem~\ref{thm-descriptions-of-G} that the images $\gbar_{a4}$ of the $g_{a4}$ generate the factors of a
  free product $(\Z/2)^{*4}\sset\PGL_2(\Bbbk)$.  It follows that the same
  holds for the $g_{a4}$ as elements of $\Aut(S)$.  The rest of the theorem is obvious.
\end{proof}

\begin{remark} 
  The fact that the subgroup of $\Aut(\Lambda)$ generated by the $g_{ab}$ is isomorphic to the
  right side of \eqref{eqGroupEquality} was first proven by W.~Swartworth \cite{Swartworth}.  He worked
  algebraically inside  $W\semidirect\frakS_5$, where $W$ is the Coxeter group
  generated by the reflections in all ten $\alpha_{ab}$.  
  Theorem~\ref{thm-descriptions-of-G} is harder than his argument, 
   but also yields
  the injectivity of
  $\Aut(S)\to\PGL_2(\Bbbk)$.
  
  Theorem~\ref{ThmAutGp} is an analogue of I.~Shimada's 
  calculation of $\Aut(S)$ when $S$ is  the Enriques surface 
  arising from a general Sylvester non-degenerate cubic surface \cite[Theorem 1.2]{Shimada}.  Shimada used 
  Borcherds' method, which is more technical but similar in flavor to ours.  
  He showed that $\Aut(S)$ is generated by the  $g_{ab}$, with defining relations
$$g_{ab}^2 = 1, \ (g_{ab}g_{bc}g_{ca})^2 = 1,\ (g_{ab}g_{cd})^2 = 1$$
where $\{a,b\}\cap \{c,d\} = \emptyset$ in the last relation.
It is an interesting question whether our methods could be adapted to recover his result, and
possibly even compute $\Aut(S)$ in the case of a arbitrary Sylvester nondegenerate cubic surface.

The group in theorem~\ref{ThmAutGp} also arose  in \cite{Mukai} as the group of automorphisms of an 
Enriques surface whose K3-cover is a quartic surface given by equation 
$s_2^2-ts_4 = 0, t\ne 0,4,36$, where $s_i$ denote elementary symmetric functions in variables 
$t_0,t_1,t_2,t_3$. These surfaces belong to a larger family given by 
the equation $t_1s_2^2+t_2s_4+t_3s_1s_3 = 0$, that includes our 1-parameter family when 
$t_1 = 0$. 
It was stated in \cite{DolgachevSalem} that the proof from \cite{Mukai} applies in our case, but
its authors have informed us that it does not. The analog of the polytope $P$ in their case is defined by the Coxeter diagram equal to the dual graph of $(-2)$-curves on Kondo's surface of type V with the same group of automorphisms isomorphic to $\frakS_5$. The surface is the limit (in the
appropriate sense) in the family when $t_2\to \infty$ (see Remark 2.3 in \cite{Mukai}).
\end{remark}

\begin{remark} We owe this remark to Matthias Sch\"utt. Consider the family of  quartic surfaces in $\bbP^3$  given by the equation
$$H_t:-ts_4(x,y,z,w)+s_1(x,y,z,w)s_3(x,y,z,w) = 0,$$
where $s_k(x,y,z,w)$ are elementary symmetric polynomials of degree $k$ 
in $x,y,z,w$. If  the characteristic $p$ of $\Bbbk$ is not equal 
to $2$ or $3$, the  surface  $H_t$  is isomorphic to the Hessian surface of a 
cubic surface with Sylvester coordinates $(1,1,1,1,t)$ 
considered in  Theorem \ref{ThmAutGenerators}. When $p= 2,3$ 
it is not the Hessian of a cubic surface, but it still
contains 10 lines $L_{ab}$ and ten nodes $P_{ab}$ 
forming the symmetric configuration $(10_3)$. The standard cubic 
Cremona involution $\tau:(x:y:z:w) \mapsto (1/x:1/y:1/z:1/w)$ leaves 
invariant each surface in the family. If $p = 2$, the 
surfaces $S_t$ with $t\ne 0,\infty$ are nonsingular, the involution acts freely and the 
quotient $S_t = H_t/(\tau)$ is an Enriques surface. There are no Coble surfaces 
in the family. If $p = 3$, the only singular surface in the family is the surface 
$S_1$. It has one  singular point $(1:1:1:1)$ and leads to a Coble surface with 
finite automorphism group isomorphic to $\frakS_5$. 
The proofs of Theorems \ref{ThmAutGenerators} and~\ref{ThmAutGp} 
and Corollary~\ref{CorOrbitsOnCurves}
extend to these families of surfaces. 
Reinterpreted in terms of the surfaces $H_t$, these results
therefore apply in all characteristics.
\end{remark}

\section{A model of $\Aut(S)$ as a lattice in $\PGL_2(\bbQ_3)$}
\label{SecLattice}

This section studies $\Aut(S)$, where $S$ is the Coble surface arising from
parameters $(\lambda_0,\dots,\lambda_4)=(1,1,1,1,\frac1{16})$, under the additional hypothesis that
$\cha\Bbbk=0$.  In this case there is one new node, so the unique bi-anticanonical divisor
is a smooth rational curve $C$, yielding a homomorphism $\Aut(S)\to\Aut(C)\iso\PGL_2(\Bbbk)$.  
We completely describe this homomorphism and its image.  The main point is that
this map is faithful.
A side benefit is that this special case is enough to identify 
the automorphism group of the surface, even when the parameters and ground field characteristic
are relaxed to $(1,1,1,1,t\neq1)$
and $\cha\Bbbk\neq2,3$. (See the proof of Theorem~\ref{ThmAutGp}.)  

Recall from Theorem~\ref{ThmAutGenerators} that $\Aut(S)$ is generated by ten involutions $g_{ab}$,
where the subscripts vary over the $2$-element subsets of $\{0,\dots,4\}$.
All we will need to know about them is the following:
\begin{enumerate}
  \item they are involutions;
  \item 
    \label{LabNontriviality}
    they act nontrivially on $C$;
  \item two commute if their corresponding pairs are disjoint;
  \item if $a,b\neq4$ then conjugation by $g_{ab}$ permutes
    the ten involutions by acting on subscripts by the transposition $(ab)$;
  \item these six $g_{ab}$ generate a copy of $\frakS_4$.
\end{enumerate}
To prove \eqref{LabNontriviality},
choose a plane $\Pi$ in $\bbP^3$ containing $P_{ab}$ and the new node $x$, such that the
tangent cone of $\Pi\cap H$ consists of two lines.  
By its definition, $g_{ab}$ exchanges them.  
So it acts by a nonscalar on the tangent cone to $H$ at~$x$, hence nontrivially on its
projectivization~$C$.

We write $\gbar_{ab}$ for the image of $g_{ab}$ in $\Aut C$, and $\Gbar$ for the subgroup 
of $\Aut C$  generated by the ten $\gbar_{ab}$.  
Our first description of $\Gbar$ is as a subgroup of $\Aut\bbH$, where $\bbH$ is Hamilton's quaternion
algebra over $\bbQ$.
  By ``norm'' we mean the reduced norm in the
sense of division algebras: $1$, $i$, $j$ and~$k$ are orthogonal unit vectors.  Because $\Bbbk$
is algebraically closed, $\Bbbk\tensor\bbH$ is isomorphic to the $2\times2$ matrix algebra $M_2(\Bbbk)$.
Therefore, any subgroup of 
$\Aut\bbH$  may
be regarded as a subgroup of $\Aut M_2(\Bbbk)\iso\PGL_2(\Bbbk)$.
We begin with a complete description of the homomorphism $\Aut S\to\Aut C$, in terms of  
an integral form of $\bbH$ called the Hurwitz integers and written $\hur$.
It is defined as the $\Z$-span of the $24$ unit norm quaternions $\pm1$, $\pm i$,
$\pm j$, $\pm k$ and $\frac12(\pm1\pm i\pm j\pm k)$.

\begin{thm}\label{thm-F-equals-Gbar}
There is an isomorphism $C\iso\bbP^1$ under which
 the four $\gbar_{a4}$ correspond to
the images in $\PGL_2(\Bbbk)$ of 
$\pm i\pm j\pm k$
and the six remaining $\gbar_{ab}$ 
correspond to the images in $\PGL_2(\Bbbk)$  
of $\pm i\pm j$, $\pm j\pm k$ and $\pm k\pm i$.
\end{thm}

\begin{proof}
  Consider the convex hull in $\Im(\bbH\tensor\bbR)=\bbR^3$ of $\pm i\pm j\pm k$.  
It is a cube centered at the origin.
The twelve Hurwitz
integers $\pm i\pm j$, $\pm j\pm k$ and $\pm k\pm i$ are the midpoints of its edges. 
%\marginpar{\ID{Strictly speaking the midpoint does not belong to the graph  but only to its topological realization}} 
The
conjugation action on $\Im\bbH$ of any one of them is the 
order~$2$ rotation that fixes that  midpoint.  This proves
that the subgroup of $\Aut\bbH$ generated by their conjugation actions 
is the rotation group of the cube.  In particular, it is isomorphic to $\frakS_4$.

Now consider the subgroup $\frakS_4\sset\Aut(S)$ 
generated by the $g_{ab}$ with $a,b\neq4$.  It is easy to check that it
acts faithfully on the tangent space to $\bbP^3$ at the node 
$(1,1,1,1,-4)$ of the Hessian surface $H$.  
It follows that the nontrivial (hence noncentral) 
elements of $\frakS_4$ act on this tangent space by non-scalars,
which implies that
$\frakS_4$ acts faithfully
on the projectivized tangent cone to $H$ at that node.  That is, $\frakS_4$ acts faithfully on $C$.

$\PGL_2(\Bbbk)$ contains a unique conjugacy class of subgroups isomorphic to~$\frakS_4$.
So $C$ may be identified with $\bbP^1$ in such a way that 
these two groups $\frakS_4$ are identified.
Under this identification, the six involutions outside $\frakA_4$,
namely the $\gbar_{ab}$ with $a,b\neq4$, correspond to the six rotations of the cube
considered above.  That is, these $\gbar_{ab}$ are 
identified with the images in $\PGL_2(\Bbbk)$ of $\pm i\pm j$, $\pm j\pm k$ and $\pm k\pm i$.

By the hypothesis on how 
$\frakS_4\sset\Aut(S)$ permutes $\gbar_{04},\dots,\gbar_{34}$, we know that
some $\gbar_{a4}$ centralizes each order~$3$ subgroup of the rotation group of the 
cube. For example, $\gbar_{04}$ centralizes the subgroup generated by $\gbar_{12}\circ \gbar_{23}$.
%\marginpar{\ID{Added a sentence}}  
The order~$3$ subgroups are generated by the order~$3$ rotations around the
body-diagonals of the cube.  Only one order~$2$ 
element of $\PGL_2(\Bbbk)$ centralizes any given order~$3$ element of $\PGL_2(\Bbbk)$.  
In our case it  is easy to
exhibit: the order~$2$ 
rotation around that same body-diagonal.
That is, by the conjugacy action of one of ${ }\pm i\pm j\pm k$.  Therefore the
$\gbar_{04},\dots,\gbar_{34}$ act as stated.
%(By the action of $\frakS_4$, we may choose the correspondence between $\gbar_{04},\dots,
 % \discretionary{}{}{}\gbar_{34}$
  %and the body diagonals to be any given bijection.  Once this is fixed, 
  %the remaining six $\gbar_{ab}$ are determined by their action on the body-diagonals.
  %So the statement of the theorem determines the homomorphism $\Aut(S)\to\PGL_2(\Bbbk)$ up
%to conjugacy.)
\end{proof}
%\marginpar{\ID{deleted a paragraph because of the following corollary.}}

%\marginpar{\ID{Added corollary allowing me to delete the last section}}
\begin{corollary}
  \label{tetrahedron} 
  The image $\bar{G}$ of the restriction homomorphism $\Aut(S)\to \Aut(C)$ is conjugate to the 
  subgroup of $\SO(3)$ generated by rotations $r_{ab},0\le  a,b\le 3$ around the midpoints of the
  edges of a cube, and the order 2 rotations $r_{a4}$ around the four body diagonals. 
  \qed
\end{corollary}

 %\begin{proof} The first assertion follows from the proof. In fact, when we fix the correspondence between $\gbar_{04},\dots,
 % \discretionary{}{}{}\gbar_{34}$ and  the body diagonals,  the remaining six $\gbar_{ab}$ are determined by their action on the body-diagonals. 
 %\end{proof} 
 
\begin{remark}
  \label{RemMatrices}
  One can write down explicit matrices for the $\gbar_{ab}$ by choosing an
  isomorphism $\hur\tensor\Bbbk\iso M_2(\Bbbk)$.  For example, if $\Bbbk=\bbC$ then one
  standard isomorphism is
  \begin{equation*}
  i\leftrightarrow\begin{pmatrix}0&-\bfi\\-\bfi&0\end{pmatrix}
  \qquad
  j\leftrightarrow\begin{pmatrix}0&-1\\1&0\end{pmatrix}
  \qquad
  k\leftrightarrow\begin{pmatrix}-\bfi&0\\0&\bfi\end{pmatrix}
  \end{equation*}
  where $\bfi=\sqrt{-1}\in\bbC$.
  So the matrices for the $\gbar_{ab}$ with $a,b\neq4$ are the signed sums of pairs
  of these matrices, and the matrices for the $\gbar_{a4}$ are the signed sums of all three.
  (This gives $20$ matrices, but only $10$ up to sign.) By direct computation of the restriction homomorphism for the Coble surface, the second author was able, using MAPLE, to find 10 matrices corresponding to $\bar{g}_{ab}$ that generate a group conjugate  to the group generated by the symmetries of a regular tetrahedron and reflections across its facets. The present proof gives a more elegant and non-computational proof of this result. 
  %\marginpar{\ID{Added a paragraph}.}
\end{remark}

%It follows from the proof of the theorem that $\Gbar$ is the subgroup
%of $\SO(3)$ generated by the rotation group of the cube and the order~$2$ rotations
%around its four body-diagonals.  
Although pretty, Corollary~\ref{tetrahedron}'s description of $\Gbar$ is difficult
to use because $\Gbar$ is not discrete in $\SO(3)$.  Our next result realizes
$\Gbar$ as a 
discrete group in $\PGL_2(\bbQ_3)$ rather than $\SO(3)$.  Here $\bbQ_3$ is the field
of $3$-adic rational numbers.  
The embedding 
$\Gbar\to\PGL_2(\bbQ_3)$ arises as follows.  

We write $\bfF$ for the algebraic group scheme over $\Z$, which to each commutative ring
$R$ assigns the group $\Aut(\hur\tensor R)$. 
This is just a $\Z$-form of  $\PGL_2$, in the sense that the functor becomes equal to $\PGL_2$
after base changing to any
field over which the division algebra $\bbH$ splits.
%\marginpar{\ID{Added a sentence. Please check whether it makes sense}} 
We claim that $\Gbar\sset\bfF(\Z[\frac13])$.
To see this, note that ${ }\pm i\pm j\pm k$ have norm~$3$ and so their inverses
lie in $\hur\tensor\Z[\frac13]$.  
So their conjugation maps lie in $\bfF(\Z[\frac13])$.
The inverses of ${ }\pm i\pm j$, ${ }\pm j\pm k$ and ${ }\pm k\pm i$
do not lie in $\hur\tensor\Z[\frac13]$.  But the conjugation maps of these 
Hurwitz integers do preserve $\hur$,
hence lie in $\bfF(\Z)$.  So $\Gbar\sset\bfF(\bbZ[\frac13])$.
The embedding $\Z[\frac13]\to\bbQ_3$ induces an inclusion $\bfF(\bbZ[\frac13])\to\bfF(\bbQ_3)$.
Since $\bbH$ splits over $\bbQ_3$, we have $\bfF(\bbQ_3)\iso\PGL_2(\bbQ_3)$.  Putting
all of this together yields an embedding $\Gbar\to\PGL_2(\bbQ_3)$.

\begin{lem} 
  \label{LemDiscreteness}
  The groups $\bfF(\bbZ[\frac13])$ and $\Gbar$
  are discrete in  $\bfF(\bbQ_3)\iso\PGL_2(\bbQ_3)$.
\end{lem}

\begin{proof}
  The discreteness of $\bbZ[\frac13]$ in $\bbR\times\bbQ_3$ shows that $\bfF(\bbZ[\frac13])$
  is discrete in $\bfF(\bbR)\times\bfF(\bbQ_3)$.  Projecting onto the second factor
  preserves discreteness because $\bfF(\bbR)\iso\SO(3)$ is compact.
\end{proof}

Over any $p$-adic field, $\PGL_2$ acts properly on a certain tree, which is 
a standard tool for working with discrete subgroups.  We recall the construction 
over $\bbQ_3$.
This tree $\calT$ has vertex set equal to the set of homothety classes of lattices (rank two
$\Z_3$-submodules) in $\Q_3^2$.  Note that while $\PGL_2(\Q_3)$ does not act on $\Q_3^2$,
it does act on the set of homothety classes of lattices.  Two vertices are adjacent just if
there are lattices representing them, such that one contains the other of index~$3$.  
In particular, if a lattice $L$ represents a vertex $v$ of~$\calT$, then its neighbors in 
$\calT$ correspond to the index~$3$ sublattices of $L$, or equivalently the $1$-dimensional
subspaces of $L/3L\iso\F_3^2$.  Under $\SL_2(\Q_3)$, the vertices fall into two orbits, which
are exchanged by any element of $\GL_2(\Q_3)$ of determinant~$3$.

\begin{lem}
  \label{lemBinTet}
  Let $A$ be the binary tetrahedral group consisting of the $24$ units of~$\hur$.
 Its image $\Abar$ in $\bfF(\bbZ)$ lies in~$\Gbar$.
\end{lem}

\begin{proof}
  %By theorem~\ref{xx}, $\Gbar$ contains the conjugation map of 
%  The subgroup of the multiplicative group $\bbH^*$ generated by
%  ${ }\pm i\pm j$, ${ }\pm j\pm k$ and ${ }\pm k\pm i$
%  contains 
  Observe that
  $(j-k)(i-j)^{-1} = (−1+i+j+k)/2$.  This quaternion is an order~$3$ element of~$A$,
  and its conjugation map lies in $\Gbar$ by theorem~\ref{thm-F-equals-Gbar}.  Conjugating by
  ${ }\pm i\pm j$, ${ }\pm j\pm k$ and ${ }\pm k\pm i$ shows that $\Gbar$ contains
  the conjugation maps of  all eight order~$3$ elements of~$A$,
  namely $(-1\pm i\pm j\pm k)/2$.  These generate $A$.  
  %Taking images in $\Aut\hur$ shows
  %$\Abar\sset\Gbar$.
\end{proof}

\begin{lem}\label{lem3.5}
	\label{lem-A-action-on-T}
	The binary tetrahedral group $A$ fixes a unique vertex $v$ of $\calT$.  Each order~$3$
	element of~$A$ fixes $v$, exactly one neighbor of~$v$, and no other vertices of~$\calT$.
\end{lem}

\begin{proof}
	Because $A$ is a finite subgroup of $\SL_2(\Q_3)$, it preserves some lattice~$L$, for example
	the $\Z_3$-span of the $A$-images of your favorite nonzero vector.  We write $v$ for the
	corresponding vertex of~$\calT$.  By identifying $L$ with $\Z_3^2$, we identify the $\SL_2(\Q_3)$-stabilizer
	of $L$ with $\SL_2(\Z_3)$.  Because the kernel of $\SL_2(\Z_3)\to\SL_2(\F_3)$ is a pro-$3$ group,
	the normal subgroup $Q_8$ of $A$ maps faithfully to $\SL_2(\F_3)$.  Since every $\Z/3$ subgroup
	of $A$ acts nontrivially on $Q_8$, it also maps faithfully.  Therefore the composition
	$A\to\SL_2(\Z_3)\to\SL_2(\F_3)$ is injective.  It is even an isomorphism, because $|\SL_2(\F_3)|=24$.

	We have shown that $A$ acts on $L/3L\iso\F_3^2$ as $\SL_2(\F_3)$.  In particular, it permutes
	the four $1$-dimensional $\F_3$-subspaces as the alternating group of degree~$4$.  It follows that
	$A$ fixes no neighbor of $v$ (hence no point of $\calT$ other than~$v$), and that each order~$3$
	element of~$A$ fixes exactly one neighbor of~$v$.  

	It remains to show that no vertex at distance~$2$ from~$v$ is fixed by any order~$3$ element of~$A$.
	Each such vertex is represented by a lattice $M$ having index~$9$ in $L$.  Furthermore, $L/M$
	cannot be isomorphic to $(\Z/3)^2$, because that would force $M=3L$, which corresponds to the
	vertex~$v$ rather than to a vertex at distance~$2$.  Therefore $L/M\iso\Z/9$.  It follows that the
	vertices at distance~$2$ correspond to the $12$ subgroups $\Z/9$ of $L/9L\iso(\Z/9)^2$.  It
	suffices to show that no order~$3$ element of $A$ preserves any one of them.  This follows from
	the claim: every order~$3$ element of $\SL_2(\Z/9)$, that preserves some $\Z/9\sset(\Z/9)^2$,
	lies in the kernel of $\SL_2(\Z/9)\to\SL_2(\F_3)$.

	To prove the claim, we use the fact that all the $\Z/9$'s are $\SL_2(\Z/9)$-equivalent, so
	it is enough to examine the order~$3$ elements in the stabilizer of the $\Z/9$ 
  generated by $\bigl(\begin{smallmatrix}1\\0\end{smallmatrix}\bigr)$.  
  This stabilizer is the semidirect product
  $\gend{\tau}\semidirect\gend{\sigma}$, where
  $\tau=\bigl(\begin{smallmatrix}1&1\\0&1\end{smallmatrix}\bigr)$ has order~$9$ and
   $\sigma=\bigl(\begin{smallmatrix}2&0\\0&1/2\end{smallmatrix}\bigr)$ has order~$6$.
One can check that 
	$\sigma\tau\sigma^{-1}=\tau^4$.  We must show that every order~$3$ element 
  $x\in\gend{\tau}\semidirect\gend{\sigma}$
	has trivial image in $\SL_2(\F_3)$.  It is clear that every order~$3$ element lies in 
$\gend{\tau}\semidirect\gend{\sigma^2}$.  
%By replacing $x$ by $x^{-1}$ we may suppose $x=\tau^i\sigma^{0\,{\rm or}\,2}$.
If $x=\tau^i$ then the relation $x^3=1$ forces $3|i$.  If $x=\tau^i\sigma^{\pm2}$ then the
relation $x^3=1$ boils down to $\tau^{273i}=1$, which again forces $3|i$.
We have proven
 $x\in\gend{\tau^3}\semidirect\gend{\sigma^2}$.  
 This implies our claim, because $\tau^3$ and $\sigma^2$ map to the identity of $\SL_2(\F_3)$.
  \end{proof}

\begin{thm}\label{thm3.6}
	\label{thm-descriptions-of-G}
	\leavevmode
	\begin{enumerate}
		\item
			\label{item-four-Z-mod-2s}
      The subgroup of $\Aut(C)\iso\PGL_2(\Bbbk)$ generated by $\gbar_{04},\dots,\gbar_{34}$
				is the free product
				$\gend{\gbar_{04}}*\cdots*\gend{\gbar_{34}}$ of four copies of~$\Z/2$.  It acts simply
				transitively on the vertices of~$\calT$.
		\item
			\label{item-description-of-G}
      The image
      $\Gbar$ of $\Aut(S)$ in $\Aut(C)\iso\PGL_2(\Bbbk)$ is the semidirect product
      of the group from \eqref{item-four-Z-mod-2s} by the symmetric group~$\frakS_4$,  
      permuting the free factors $\Z/2$ in the obvious way.
	\end{enumerate}
  In particular,
 the natural map $\Aut(S)\to\Aut(C)$ is injective.
\end{thm}

\begin{proof} 
  Even though $\Gbar$ was defined as a subgroup of $\Aut(C)\iso\PGL_2(\Bbbk)$, 
  we will continue to work with it as a subgroup of $\PGL_2(\bbQ_3)$.  
  We continue to write $v$ for the unique vertex of $\calT$ fixed by the binary tetrahedral group~$A$.
  Each of ${ }\pm i\pm j\pm k$ has norm~$3$ in $\hur$, 
  hence determinant~$3$ when regarded as an element of $\GL_2(\Q_3)$.
  Therefore each $\gbar_{a4}$ exchanges the two $\SL_2(\Q_3)$-orbits of vertices of~$\calT$.  
  In particular, $\gbar_{a4}$
	moves $v$  to some other vertex.  
  
  Next, $\gbar_{a4}$ centralizes
  some  order~$3$ subgroup $\Theta_a$ of $\frakS_4\sset\Gbar$.
  The image of $A$ in $\PGL_2(\bbQ_3)$ 
  contains all the order~$3$ subgroups of $\frakS_4$.  Therefore
  Lemma~\ref{lem-A-action-on-T} shows that $\Theta_a$ fixes the vertex $v$ and one of its neighbors,
  but no other vertices of~$\calT$.
  Since $\gbar_{a4}$ centralizes $\Theta_a$, it preserves this set of two vertices.
  In the previous paragraph we saw that $\gbar_{a4}$ moves $v$ to some other vertex.
  Therefore $\gbar_{a4}$ exchanges $v$ with the neighbor fixed by $\Theta_a$.
  The midpoint $m_a$ of the edge joining these vertices is the only fixed point of $\gbar_{a4}$ in
  $\calT$.  
  We think of $\gbar_{a4}$ as acting on $\calT$ by a sort of reflection, whose mirror 
  consists of the single point~$m_a$.
  Each of $\gbar_{04},\dots,\gbar_{34}$ centralizes a different order~$3$
  subgroup of $\frakS_4$, so $m_0,\dots,m_3$ are the midpoints of the four edges emanating
  from~$v$.

  This suggests that the union $D$
  of the four half-edges from $v$ to $m_0,\dots,m_3$
  should be a fundamental domain for the action of $\gend{\gbar_{04},\dots,\gbar_{34}}$
  on~$\calT$.
	This can be verified by using Poincar\'e's Polyhedron Theorem. 
  Standard references, such as  \cite[sec. IV.H]{Maskit}, only develop this theorem for groups acting
on manifolds. So we sketch the proof in our situation, which is actually 
much simpler than the general manifold case.
  
  We set $H=\gend{\gbar_{04},\dots,\gbar_{34}}$ and
  \begin{equation*}
  \Htilde=\gend{\gbar_{15}}*\cdots*\gend{\gbar_{45}}
  \iso(\Z/2)*(\Z/2)*(\Z/2)*(\Z/2)
\end{equation*}
We write elements of $\Htilde$ with
tildes (except for the $\gbar_{a4}$), and indicate
  the natural map $\Htilde\to H$ by removing the tilde.
  We think of $\Htilde\times D$ as a
   disjoint union of copies of 
   $D$ indexed by $\htilde\in\tilde{H}$.  
   There is a natural $\Htilde$-action on this union, with $\gtilde\in\Htilde$
   sending $(\htilde,d)$ to $(\gtilde\htilde,d)$. The map $\Htilde\times D\to\calT$ defined by
   $(\htilde,d)\mapsto h(d)$ is compatible with the natural map $\Htilde\to H$
   and the $\Htilde$- and $H$-actions on $\Htilde\times D$ and $\calT$.  We glue the copies of $D$ together 
   to form a connected graph $\calTtilde$, by identifying $(\htilde,m_a)$ with $(\htilde\gbar_{a4},m_a)$,
   for every $\htilde\in\Htilde$ and $a=0,\dots,3$.  
   The gluing is compatible with
   the $\Htilde$-action, so $\Htilde$ acts on $\calTtilde$. 
   The gluing is also compatible with 
   the projection $\Htilde\times D\to\calT$,
   which therefore descends to a map $\calTtilde\to\calT$.  This map is compatible with $\Htilde\to H$ 
   and the
   $\Htilde$- and $H$-actions.
   It is easy to check that $\calTtilde\to\calT$ is 
	a covering map, hence a homeomorphism.  It follows that $\tilde{H}\to H$ must be an isomorphism.
  The simple transitivity of $\Htilde$ on the vertices $(\htilde,v)$ of $\calTtilde$ is
  obvious, so $H$ acts simply transitively on the vertices of $\calT$.

\eqref{item-description-of-G} 
Having proven \eqref{item-four-Z-mod-2s}, we know that $\Gbar$ is generated by $\frakS_4$ and
the free product of four copies of
$(\Z/2)$, with the first group normalizing the second.  To establish the
semidirect product decomposition we must show that these groups meet trivially.
As a finite subgroup of the free product, the intersection has order${ }\leq2$.  But the 
intersection is also normal in $\frakS_4$, which has no normal subgroups of order~$2$.  Therefore
the intersection is trivial.
\end{proof}

\begin{thm}
  \label{ThmGpFunctor}
  The group $\Gbar$ coincides with $\bfF(\bbZ[\frac13])$, which is 
 maximal among discrete subgroups of $\PGL_2(\bbQ_3)$.
\end{thm}

\begin{proof}
  It suffices to show that $\Gbar$ is maximal among discrete subgroups of $\PGL_2(\bbQ_3)$. So
   suppose $\Gamma$ is a discrete subgroup 
  that contains it.  Because $\Gbar$ acts transitively on the vertices of $\calT$,
  $\Gamma$ is generated by $\Gbar$ and the $\Gamma$-stabilizer of~$v$.
  The latter is finite, by discreteness.  It contains the $\Gbar$-stabilizer $\frakS_4$ of~$v$.
  Since $\frakS_4$ is maximal among finite subgroups of $\PGL_2$, over any field of characteristic~$0$,
  the  $\Gamma$-stabilizer is the same as the $\Gbar$-stabilizer.  So $\Gamma=\Gbar$.
\end{proof}

%\begin{remark}\label{rmk3.7}
%  Using the action on $\calT$,
%	one can also show that $(\hur\tensor\Z[\frac13])^*_1$ is generated by $A$ and the pairwise quotients
%	of the
%  ${ }\pm i \pm j\pm k$.  
%  In fact these pairwise quotients generate a normal subgroup which is free of rank~$3$,
%	so $(\hur\tensor\Z[\frac{1}{3}])^*_1$ is the semidirect product of this free group by the binary tetrahedral
%	group.  Furthermore, its image in $\PGL_2(\Q_3)$ has index~$4$ in its $\PGL_2(\Q_3)$-normalizer,
%	which turns out to be the group $\Gbar$.
%  This realizes $\Gbar$ as a specific arithmetic lattice in $\PGL_2(\bbQ_3)$.
%\end{remark}

Theorem~\ref{thm-descriptions-of-G} has an appealing consequence:

\begin{thm}
  \label{thm3.8}
	\label{thm-tetrahedron-reflections}
  Let $T$ be a regular tetrahedron in Euclidean $3$-space $\R^3$.  Then the group of
  isometries generated by the automorphisms of $T$ and the reflections across its facets
  is 
  \begin{equation*}
    \bigl( (\Z/2)*(\Z/2)*(\Z/2)*(\Z/2) \bigr)\semidirect\Aut(T)
  \end{equation*}
\end{thm}
%\marginpar{\ID{I changed the proof using the new Corollary 5.2. Please check}}
%However, the reflection group
%in that sitation is solvable, since the isometry group of Euclidean $2$-space is.
%Therefore your friend would have many solutions 
%available besides retracing your sequence of reflections.

\begin{proof} 
  It suffices to show that the image of this group in $\orth(3)$ has this structure.
  Combining  Corollary \ref{tetrahedron} and Theorem~\ref{thm-descriptions-of-G} 
  shows that the group $\Gbar\sset\SO(3)$
  generated
  by the rotations $r_{ab}$  $(a,b=0,\dots,3)$ around the midpoints of the edges of the cube,
  and the rotations $r_{a4}$ around the body diagonals of the cube, has this structure.
  Replacing each $r_{ab}$ by $-r_{ab}$ replaces each $g\in\Gbar$ by $\pm g$, and therefore
  does not change the isomorphism type of the subgroup of $\orth(3)$ they
  generate. 
  
  We identify $T$ with one of the two regular tetrahedra inscribed in the cube. 
  Then the planes through the origin, perpendicular to the body diagonals, are parallel to the 
  facets of $T$.  Therefore the $-r_{a4}$ are the reflections across them. And the
  $-r_{ab}$ with $a,b\le 3$ generate the automorphism group of $T$.
\end{proof}

An amusing way to interpret this is that you can reflect $T$ across a facet, and then reflect that image
of $T$ across one of its facets, and so on.  Imagine doing this and then challenging your
friend to return the tetrahedron to its original position by further reflections. The
only solution is to retrace your sequence of reflections.  This is a $3$-dimensional
version of Rich Schwartz's game ``Lucy and Lily'' \cite{Schwartz}, which uses 
a regular pentagon in the plane in place of our tetrahedron.  

In fact, the entire paper grew backwards from theorem~\ref{thm-tetrahedron-reflections}.  
The explicit matrix computations
referred to in remark~\ref{RemMatrices} identified $\Gbar$ with the group from this theorem.  The problem
of identifying the image led to the discrete subgroup of $\PGL_2(\bbQ_3)$.  
Then, having identified $\Gbar$, it
was natural to wonder whether $\Aut(S)\to\Gbar$ was faithful.  And this led to the computation of the
nef cone.
 
%\ID{With great sorrow I decided to elimitate the last section. I will keep my good memory about these computations that gave rise to our paper but they really became obsolete.}

\end{document}